\documentclass[12pt, reqno]{amsart}
\usepackage{amsmath,amssymb,amsthm,amsxtra, setspace}
\usepackage{amssymb}
\usepackage{stmaryrd}
\usepackage{alltt}
\usepackage{graphicx,type1cm,xcolor}
\usepackage{mathtools, accents}
\usepackage{mathrsfs}
\usepackage{hyperref}
\usepackage{alltt}
\usepackage{fouridx}
\usepackage{dsfont}
\usepackage{dsfont}
\usepackage{cancel}
\usepackage[mathcal]{euscript}
\usepackage[margin=1in]{geometry}
\usepackage{shadethm}
\usepackage{float}

\allowdisplaybreaks

\newcommand{\norm}[1]{\left\|#1\right\|}

\numberwithin{equation}{section}

\newshadetheorem{Theorem}{Theorem}
\newshadetheorem{Proposition}{Proposition}
\newshadetheorem{Lemma}{Lemma}
\newshadetheorem{Assumption}{Assumption}
\newshadetheorem{Corollary}{Corollary}
\newshadetheorem{Definition}{Definition}
\newshadetheorem{Example}{Example}
\newshadetheorem{Remark}{Remark}
\newshadetheorem{Note}{Note}
\newshadetheorem{Hypothesis}{Hypothesis}
\newshadetheorem{Explanation}{Explanation}
\newtheorem{theorem}{Theorem}[section]
\newtheorem{lemma}[theorem]{Lemma}
\newtheorem{proposition}[theorem]{Proposition}

\newtheorem{definition}[theorem]{Definition}

\newtheorem{remark}[theorem]{Remark}

\let\originalleft\left
\let\originalright\right
\renewcommand{\left}{\mathopen{}\mathclose\bgroup\originalleft}
\renewcommand{\right}{\aftergroup\egroup\originalright}

\newcommand{\ip}[2]{\fourIdx{}{0}{}{\!x}{\mathcal{ X}}}

\newcommand\dela[1]{}
\hypersetup{colorlinks=true,%
	citecolor=red,%
	filecolor=blue,%
	linkcolor=blue,%
}
\usepackage{graphicx}
\usepackage[utf8]{inputenc}
\usepackage[T1]{fontenc}
\mathtoolsset{showonlyrefs}
\usepackage{nomencl}
\makenomenclature
\usepackage{amssymb}

\usepackage[symbol]{footmisc}
\usepackage{centernot}
\usepackage{ragged2e}
\justifying
\usepackage{orcidlink}

\usepackage{todonotes}

\usepackage{xpatch}
\makeatletter   
\xpatchcmd{\@tocline}
{\hfil\hbox to\@pnumwidth{\@tocpagenum{#7}}\par}
{\ifnum#1<0\hfill\else\dotfill\fi\hbox to\@pnumwidth{\@tocpagenum{#7}}\par}
{}{}
\makeatother    
\makeatletter
\def\l@subsection{\@tocline{2}{0pt}{4pc}{6pc}{}}
\def\l@subsubsection{\@tocline{3}{0pt}{8pc}{8pc}{}}
\makeatother

\makeatletter
\def\l@section{\@tocline{1}{12pt}{0pt}{}{\bfseries}}% <- added
\makeatother
\usepackage{enumitem}

\title[Constrained Kuramoto--Sivashinsky Equation]{Global Well-Posedness of the Kuramoto--Sivashinsky Equation in One and Two Dimensions under a Codimension-One Constraint}

\begin{document}
	\maketitle
	\begin{center}
		\author{Harsh Tiwari\footnote[4]{Department of Mathematics, Indian Institute of Technology Roorkee-IIT Roorkee, Haridwar Highway, Roorkee, Uttarakhand 247667, India.\\
				\textit{e-mail:} H. Tiwari: \email{harsh\_t@ma.iitr.ac.in; harshtiwari98765@gmail.com}.\\
				\textit{e-mail:} Manil T. Mohan: \email{maniltmohan@ma.iitr.ac.in; maniltmohan@gmail.com}.\\
				$\hspace{2mm} ^\ast$Corresponding author.\\
				\textit{MSC 2020}:  35B40,  	%Asymptotic behavior of solutions to PDEs
				35L15  	%Initial value problems for second-order hyperbolic equations
				35L70,  	%Second-order nonlinear hyperbolic equations
				35R01,  	%PDEs on manifolds
				58J45,  	%Hyperbolic equations on manifolds
				\\
				\textit{Key words}: 
				%Damped wave equation $\cdot$  Feado-Galerkin approximation $\cdot$ Constrained PDEs $\cdot$ Strong solutions $\cdot$ Asymtotic precompactness $\cdot$  Asymtotic analysis
%				\textit{Key words}: Stochastic heat equation $\cdot$ Martingale solutions $\cdot$ Multiplicative noise $\cdot$ Constraints $\cdot$ Stochastic gradient flow $\cdot$ Pathwise uniqueness $\cdot$ It\^o formula
			}\orcidlink{0009-0004-4685-8878} and Manil T. Mohan$^\ast$\footnotemark[4]\orcidlink{0000-0003-3197-1136}}
	\end{center}

	\begin{abstract}
We study the Kuramoto--Sivashinsky (KS) equation  in one, two, and three spatial dimensions subject to a codimension-one constraint. The constraint preserves the \(L^2\)-norm of the solution and consequently provides a uniform a priori bound that plays a central role in the higher-order energy estimates. In one spatial dimension, we establish global well-posedness by means of a Fourier--Galerkin approximation and compactness arguments. In two and three dimensions, we first prove the existence and uniqueness of local strong solutions using analytic semigroup theory and the Banach fixed-point theorem. In two dimensions, the preserved \(L^2\)-norm, combined with the Agmon inequality and suitable higher-order energy estimates, allows us to control the nonlinear terms and obtain global-in-time a priori bounds. Consequently, we establish the existence and uniqueness of global strong solutions for arbitrary sufficiently regular initial data, without any smallness assumption. This provides a global well-posedness result for the constrained two-dimensional KS equation, in contrast with the corresponding unconstrained problem, for which global well-posedness of strong solutions remains open in general. In three dimensions, although the constraint yields a uniform \(L^2\)-bound, the available interpolation and Agmon estimates are insufficient to close the higher-order energy estimate. This identifies a genuine dimension-dependent obstruction to extending the global regularity argument to three dimensions. 
	\end{abstract}

	\begin{center}
		\tableofcontents
	\end{center}
	\section{Introduction}

	The Kuramoto--Sivashinsky (KS)  equation is a fundamental nonlinear evolution equation arising in the mathematical modeling of	instability-driven phenomena, including flame-front propagation and crystal growth. The equation originated in the works of Kuramoto and
	Tsuzuki on dissipative structures in reaction--diffusion systems
	\cite{KuramotoTsuzuki1975,KuramotoTsuzuki1976} and in Sivashinsky's
	investigation of flame-front instabilities \cite{Sivashinsky1977,Sivashinsky1980}.
	Since then, the KS equation has become an important model problem in the study
	of nonlinear partial differential equations, dissipative dynamical
	systems, and chaotic dynamics. Throughout this paper, for $L>0$ and
	$d\in\mathbb N$, we denote the $d$-dimensional flat torus of side
	length $L$ by
	\begin{align*}
	\mathbb T_L^d:=(\mathbb R/L\mathbb Z)^d.
	\end{align*}
	Thus, $\mathbb T_L=\mathbb R/L\mathbb Z$ denotes the one-dimensional
	periodic domain, while $\mathbb T_L^2$ and $\mathbb T_L^3$ denote the
	two- and three-dimensional periodic domains, respectively.

	\subsection{One-dimensional KS equation}
	The scalar one-dimensional KS equation is given by
	\begin{equation}\label{eq:1d-KS equation}
		u_t+u_{xxxx}+u_{xx}
		+\frac12(u_x)^2=0,
		\ \ x\in\mathbb T_L.
	\end{equation}
	The equation exhibits a competition between the fourth-order
	dissipative term $u_{xxxx}$ and the destabilizing second-order term
	$u_{xx}$, while the quadratic gradient nonlinearity generates a
	nontrivial nonlinear interaction.  	It is often useful to introduce the derivative variable
	$
	v=u_x.
	$
	Differentiating \eqref{eq:1d-KS equation} with respect to $x$ gives
	\begin{equation}\label{eq:1d-v-KS equation}
		v_t+v_{xxxx}+v_{xx}+vv_x=0,
		\ \ x\in\mathbb T_L.
	\end{equation}
	Indeed,
	$
	\frac{	\partial}{\partial x}\left(\frac12(u_x)^2\right)
	=u_xu_{xx}=vv_x.
	$
	The variable $v$ necessarily satisfies the zero-mean condition
	\begin{align*}
	\int_{\mathbb T_L}v(x,t) dx=0,
	\end{align*}
	which follows from its definition as the spatial derivative of a
	periodic function. Conversely, $u$ can be recovered from $v$, up to an
	additive spatial constant. Thus, \eqref{eq:1d-KS equation} and
	\eqref{eq:1d-v-KS equation} provide equivalent formulations once the spatial
	mean of $u$ is fixed.

	The mathematical theory of KS equation is particularly well developed in one spatial dimension. Global well-posedness of the one-dimensional KS equation was established in the classical works of Nicolaenko, Scheurer, and Temam and Tadmor, among others \cite{NicolaenkoScheurer1984,NicolaenkoScheurerTemam1985,Tadmor1986}. In this setting, the nonlinear term possesses a special structure that allows one to obtain effective \(L^2\)-energy estimates and consequently global control of solutions. Beyond global existence, the one-dimensional equation has been extensively investigated from the viewpoint of its long-time behavior. In the periodic setting, solutions possess a finite-dimensional global attractor, and a substantial literature is devoted to estimates for attractor dimensions, determining modes, asymptotic behavior, and related dynamical properties; see, for example, \cite{BronskiGambill2006,ColletEtAl1993,ColletEtAl1993a,FoiasNicolaenkoSellTemam1988,GiacomelliOtto2005,Goodman1994,KostiankoTitiZelik2018,NicolaenkoScheurerTemam1985,Otto2009} and the references therein.

%	Despite this competition, the
%	one-dimensional KS equation is globally well posed for sufficiently regular
%	initial data \cite{NicolaenkoScheurer1984,Tadmor1986}. Its long-time
%	dynamics also possess a remarkable finite-dimensional structure,
%	including finite-dimensional global attractors and inertial manifolds
%	\cite{FoiasNicolaenkoSellTemam1988,NicolaenkoScheurerTemam1985}. These results have made the
%	one-dimensional KS equation a fundamental model for studying dissipative
%	dynamics, nonlinear instability, and chaotic behavior.

	\subsection{Multidimensional KS equation}
	For $d\geq2$, the natural scalar generalization of
	\eqref{eq:1d-KS equation} is
	\begin{equation}\label{eq:d-KS equation}
		\varphi_t+\Delta^2\varphi+\Delta \varphi
		+\frac12|\nabla \varphi|^2=0,
		\ \  x\in\mathbb T_L^d,
	\end{equation}
	where $\varphi$ is a scalar-valued function. The multidimensional equation
	retains the competition between the fourth-order dissipative term and
	the destabilizing second-order term, but its nonlinear structure leads
	to substantially greater analytical difficulties. In particular, the
	global well-posedness problem for the periodic multidimensional KS equation
	remains open in general for arbitrary large smooth initial data in
	dimensions $d\geq2$ \cite{LariosRahmanYamazaki2022}.
	As in one dimension, one may differentiate the scalar equation by
	introducing
$
	\boldsymbol{v}=\nabla \varphi.
$
	Differentiating \eqref{eq:d-KS equation} yields
	\begin{equation}\label{eq:gradient-KS equation}
		\boldsymbol{v}_t+\Delta^2\boldsymbol{v}+\Delta \boldsymbol{v}
		+\nabla\left(\frac12|\boldsymbol{v}|^2\right)=\boldsymbol{0}.
	\end{equation}
	Since $\boldsymbol{v}=\nabla \varphi$,
$
	\nabla\left(\frac12|\nabla \varphi|^2\right)
	=
	(\nabla \varphi\cdot\nabla)\nabla \varphi
	=
	(\boldsymbol{v}\cdot\nabla)\boldsymbol{v},
$
	and hence
	\begin{equation}\label{eq:vector-KS equation}
		\boldsymbol{v}_t+\Delta^2\boldsymbol{v}+\Delta \boldsymbol{v}+(\boldsymbol{v}\cdot\nabla)\boldsymbol{v}=\boldsymbol{0}.
	\end{equation}
	It is important to emphasize that \eqref{eq:vector-KS equation} is the
	gradient formulation of the scalar equation \eqref{eq:d-KS equation}, rather
	than an arbitrary vector-valued equation. The field $\boldsymbol{v}$ satisfies the
	structural constraint
$
	\boldsymbol{v}=\nabla \varphi.
$
	In particular, in three dimensions,
$
	\nabla\times \boldsymbol{v}=\boldsymbol{0},
$
	or, equivalently,
$
	\partial_i v_j=\partial_j v_i,
	\  1\leq i,j\leq 3.
$
	
	The analytical difficulty of the multidimensional problem is already
	visible in the basic energy balance for the gradient formulation.
	For a general vector field $\boldsymbol{v}$, periodic integration by parts gives
	\begin{align*}
	\bigl((\boldsymbol{v}\cdot\nabla)\boldsymbol{v},\boldsymbol{v}\bigr)
	=
	-\frac12
	\int_{\mathbb T_L^d}
	|\boldsymbol{v}|^2\operatorname{div}\boldsymbol{v} dx.
	\end{align*}
	Consequently, the transport term does not vanish unless an additional
	condition such as
$
	\operatorname{div}\boldsymbol{v}=0
$
	is imposed. This is in contrast with the incompressible
	Navier--Stokes equations, for which
	\begin{align*}
	\bigl((\boldsymbol{v}\cdot\nabla)\boldsymbol{v},\boldsymbol{v}\bigr)=0
	\end{align*}
	under the divergence-free condition. The absence of such a
	cancellation in the KS equation is one of the principal obstacles to obtaining
	global higher-order estimates in dimensions $d\geq2$
	\cite{LariosRahmanYamazaki2022}.
	
	The difficulty is particularly pronounced because the fourth-order
	term strongly dissipates high spatial frequencies, whereas the
	second-order term destabilizes the low-frequency modes. The nonlinear
	term can interact with this large-scale instability and does not
	disappear from the higher-order energy estimates. Consequently, the
	strong fourth-order dissipation alone has not been sufficient to yield
	a general global regularity theory in higher dimensions.

	For the two-dimensional KS equation, a number of global existence results have nevertheless been obtained under additional structural assumptions. One of the first directions concerned thin or anisotropic domains, where the geometry suppresses part of the multidimensional instability and allows the problem to be treated in a manner closely related to the one-dimensional equation. Such results were initiated by Sell and Taboada \cite{SellTaboada1992}, and subsequently developed in works including Molinet \cite{Molinet2000}, Benachour, Kukavica, Rusin, and Ziane \cite{BenachourKukavicaRusinZiane2014}, and Kukavica and Massatt \cite{KukavicaMassatt2023}. In particular, the existing two-dimensional global theories have often relied on anisotropy in the domain, smallness of the initial data, or restrictions on the unstable Fourier modes.

	The size of the periodic domain plays a crucial role in this problem because it determines the number of linearly growing Fourier modes associated with the destabilizing second-order operator. When the spatial periods are sufficiently small, no linearly growing modes are present, and global existence can be established for sufficiently small initial data in suitable Wiener-algebra-based spaces; see, for example, Ambrose and Mazzucato \cite{AmbroseMazzucato2019}. A significant subsequent development was obtained by Ambrose and Mazzucato \cite{AmbroseMazzucato2021}, who considered rectangular two-dimensional domains supporting one linearly growing mode in each spatial direction. By separating the Fourier modes into low, intermediate, and high frequencies, they controlled the unstable and weakly stable modes through a Lyapunov-function argument and the high-frequency modes through Wiener-algebra estimates. This yielded global existence for sufficiently small initial data and represented the first global existence result in the two-dimensional periodic setting with a linearly growing mode in each direction.

	Another important approach has been to exploit advection and enhanced dissipation. Feng and Mazzucato \cite{FengMazzucato2022} investigated the two-dimensional KS equation with an additional advection term, while Coti Zelati, Dolce, Feng, and Mazzucato \cite{CotiZelatiDolceFengMazzucato2021} established global existence for the KS equation in the presence of a prescribed shear flow. In the latter work, assuming that the shear possesses only finitely many critical points and that linearly growing modes occur only in the direction of the shear, global existence was obtained for arbitrarily large \(L^2\) initial data, provided the dissipation is sufficiently weak. The underlying mechanism is the enhanced dissipation generated by the interaction between the shear-induced mixing and the fourth-order diffusion, which transfers energy from large unstable scales toward small scales where it can be efficiently dissipated.

	Several other aspects of the two-dimensional problem have also been investigated. These include anisotropically reduced KS models, analyticity and Gevrey regularity, and modified equations designed to recover global well-posedness while retaining important features of the original dynamics. For example, Larios and Yamazaki \cite{LariosYamazaki2020} introduced an anisotropically reduced two-dimensional KS equation and established its global well-posedness. Ioakim and Smyrlis \cite{IoakimSmyrlis2016} studied analyticity for KS-type equations in two spatial dimensions, while related works have considered modified nonlinearities and boundary conditions. More recently, Enlow, Larios, and Wu \cite{EnlowLariosWu2024} proposed an algebraically calmed version of the two-dimensional KS equation that is globally well-posed for arbitrary domain size and initial data. Moreover, they proved convergence of the modified solutions to solutions of the original KS equation on their interval of existence, with an explicit algebraic convergence rate.

	For three and higher dimensions, the available theory is substantially more limited. Biswas and Swanson \cite{BiswaSwanson2007} established short-time existence of smooth solutions in \(\mathbb R^d\), together with generalized Gevrey regularity. In the periodic setting, local well-posedness in \(H^1\) for \(d=2,3\) has also been established, while further regularity results have been developed in both scalar and vector formulations. In particular, Larios, Rahman, and Yamazaki \cite{LariosRahmanYamazaki2022} derived several conditional regularity criteria for the two- and three-dimensional KS equations. Their criteria are formulated in terms of integrability properties of the scalar solution, the associated vector field, its divergence, and individual components of the vector field and its gradient.

	Despite these developments, general global well-posedness for the standard multidimensional KS equation remains open for \(d\ge2\) in the periodic setting and in \(\mathbb R^d\). This includes, in particular, the unresolved problem of global regularity for arbitrary smooth initial data in two and three dimensions. The difficulty is not simply caused by the presence of the backward-diffusion term. Rather, the fourth-order dissipation provides strong control of high frequencies, while the main obstruction is associated with the possible growth of the large-scale modes and the absence of a suitable cancellation or maximum principle.

	Thus, the existing literature reveals a striking dimensional transition: the one-dimensional KS equation possesses a comprehensive theory of global well-posedness and long-time dynamics, whereas in dimensions \(d\ge2\) global solvability is known only in particular regimes, such as thin or anisotropic domains, small-data settings, restricted unstable-mode configurations, or suitably modified/advection-enhanced models. At the same time, local existence, Gevrey and analytic regularity, conditional regularity criteria, and various dynamical properties have been established for multidimensional equations. This gap between the strong dissipative character of the KS equation and the absence of a general global theory continues to motivate the development of new analytical methods for multidimensional KS equations.

	These developments
	suggest that introducing suitable additional structure into the
	multidimensional KS equation may provide a way to recover coercive estimates
	that are unavailable for the unconstrained equation. This motivates
	the study of constrained KS equation models, where the constraint is designed
	to modify the energy structure and provide additional control of the
	nonlinear terms, with the ultimate goal of establishing global
	well-posedness.
	
		\subsection{Constrained PDEs}  A broad class of constrained evolution equations has been developed in which the dynamics are required to remain on prescribed manifolds, most notably on spheres associated with conserved or prescribed norms. Such constraints arise naturally in geometric flows, nonlinear diffusion equations, fluid models, and stochastic partial differential equations. In the deterministic parabolic setting, Rybka \cite{PR-06} studied heat flows on manifolds and their long-time behavior. Caffarelli and Lin \cite{LC+FL-09} considered nonlocal heat flows incorporating an $L^2$-energy-preserving mechanism and established well-posedness together with convergence results in the setting of singularly perturbed systems. The analysis of $L^2$-norm-preserving dynamics was further developed by Ma and Cheng \cite{LM+LC-13}, who investigated global existence, stability, asymptotic behavior, and gradient estimates for such flows on manifolds. Brze\'zniak and Hussain \cite{ZB+JH-24} studied constrained nonlinear heat equations and established global well-posedness, invariance of the $L^2$-sphere, and an underlying gradient-flow structure.
		
		These ideas have also been developed for nonlinear parabolic equations involving nonlocal terms specifically designed to preserve the $L^2$-norm. Antonelli et al. \cite{PA+PC+BS-24} considered such equations both on bounded domains and in the whole space, proving local and global well-posedness and obtaining detailed information on their long-time behavior. In particular, they established convergence to stationary solutions and, for positive initial data on a ball, strong convergence to the corresponding ground state. Shakarov \cite{BS-25} subsequently investigated another class of nonlinear parabolic equations with a nonlocal $L^2$-norm-preserving term, focusing on well-posedness and convergence toward stationary states. Constrained nonlinear heat equations with polynomial nonlinearities were further studied by Bawalia et al. \cite{AB+ZB+MTM}, who proved global existence and uniqueness of strong solutions, invariance of the $L^2$-unit sphere, and strong convergence of positive solutions to the unique positive ground state. In a related direction, Bawalia et al. \cite{AB+ZB+MTM+PR} incorporated the \L{}ojasiewicz--Simon gradient inequality into the analysis of nonlinear heat equations on Hilbert manifolds with finite-codimensional constraints, providing a framework for studying convergence toward stationary states; see also \cite{PR-06}. Brze\'zniak et al. \cite{ZB+GD+MM-18} investigated constrained two-dimensional Navier--Stokes equations on $\mathbb{R}^2$ and $\mathbb{T}^2$, with particular emphasis on global well-posedness and the vanishing-viscosity limit. More recently, \cite{ABH-26} considered a sphere-constrained modified Swift--Hohenberg equation and established global well-posedness, energy dissipation, the existence of restricted global attractors, and properties of the associated equilibrium set.

		The incorporation of constraints into stochastic evolution equations leads to additional analytical challenges, since the noise must be compatible with the underlying manifold structure. Brze\'zniak and Hussain \cite{ZB+JH-20} initiated the study of constrained stochastic nonlinear heat equations driven by Stratonovich-type noise and established existence and uniqueness of mild solutions in smooth bounded two-dimensional domains. This theory was subsequently extended by Bawalia et al. \cite{AB+ZB+MTM2} to arbitrary spatial dimensions and polynomial nonlinearities of order $2\leq p<\infty$. Related constrained stochastic fluid models were considered by Brze\'zniak et al. \cite{ZB+GD-21}, who studied two-dimensional stochastic Navier--Stokes equations with multiplicative Gaussian noise in Stratonovich form. They constructed martingale solutions, established pathwise uniqueness, and obtained strong solutions through a Yamada--Watanabe-type argument. 
		
		Constrained dynamics have also been investigated in the stochastic damped-wave setting, where both the geometric constraint and the singular small-mass limit have to be taken into account. In the work \cite{HT+MTM},  we established global well-posedness and investigated the long-time dynamics of an \(L^2\)-constrained damped nonlinear wave equation, including the existence of infinitely many stationary solutions and convergence of trajectories toward stationary states. Brze\'zniak and Cerrai \cite{ZB+SC-25} established existence and uniqueness for constrained stochastic nonlinear damped-wave equations perturbed by multiplicative Gaussian noise. They also investigated the asymptotic behavior of the solutions through a Smoluchowski--Kramers approximation, identifying the limiting dynamics in the regime in which the mass parameter tends to zero. Cerrai and Xie \cite{SC+MX-25} subsequently extended this framework to a more general class in which the diffusion coefficient depends additionally on the velocity variable. They established the corresponding well-posedness results and analyzed the limiting behavior as the perturbed mass vanishes.
		
		Taken together, these works illustrate that norm-preserving and manifold-valued constraints provide a unified framework for a broad range of deterministic and stochastic evolution equations. The existing literature addresses global solvability, invariance of the constraint manifold, stability, asymptotic behavior, gradient structures, attractors, and singular limits across several classes of equations. Nevertheless, the interaction between such geometric constraints and higher-order nonlinear dissipative equations presents additional difficulties, particularly when the unconstrained equation itself has a delicate global theory. This provides a natural motivation for studying constrained higher-order evolution equations within an appropriate Hilbert-manifold framework.

		\subsection{Difficulties, approaches and novelties} The KS equation arises naturally in the description of the dynamics of planar flame fronts \cite{Sivashinsky1977}, making the one- and two-dimensional settings particularly relevant from a physical perspective. In this context, periodic boundary conditions provide a natural idealization of the underlying spatial structure and are therefore adopted throughout the analysis. Although other configurations, including channel-type geometries, are also of physical interest, we restrict our attention here to problems posed on periodic domains.
		
		\subsubsection{Motivation}
		Let $(\mathcal{H},\|\cdot\|_{\mathcal{H}})$ be a Hilbert space with inner product
		$(\cdot,\cdot)_{\mathcal{H}}$, and let
		\begin{align*}
			\mathcal{M}=\{h\in\mathcal{H}:\|h\|_{\mathcal{H}}=1\}
		\end{align*}
		denote the unit sphere in $\mathcal{H}$. For $h\in\mathcal{M}$, its tangent space is
		\begin{align*}
			T_h\mathcal{M}
			=\{z\in\mathcal{H}:(h,z)_{\mathcal{H}}=0\}.
		\end{align*}
		Let $F:D(F)\subset\mathcal{H}\to\mathcal{H}$ be a vector field and consider
		\begin{equation}
			\left\{
			\begin{aligned}
				\frac{du(t)}{dt}&=F(u(t)), \ \  t>0,\\
				u(0)&=u_0.
			\end{aligned}
			\right.
		\end{equation}
		Even if this problem admits a unique global solution, the corresponding flow need not preserve
		$\mathcal{M}$, since $F$ may not be tangent to $\mathcal{M}$; that is,
	$
		F(h)\notin T_h\mathcal{M}
	$
		in general. The sphere is invariant only if
	$
		F(h)\in T_h\mathcal{M}$ whenever $h\in D(F)\cap\mathcal{M}.
	$
		To enforce this tangency condition, we project $F(h)$ onto $T_h\mathcal{M}$. Let us define
		\begin{align*}
			\pi_h(z):=z-(h,z)_{\mathcal{H}}h,
			\ \ z\in\mathcal{H},
		\end{align*}
		and set
		\begin{align*}
			\widetilde{F}(h):=\pi_h(F(h))
			=F(h)-(h,F(h))_{\mathcal{H}}h,
			\ \  h\in D(F)\cap\mathcal{M}.
		\end{align*}
		Then
	$
		(\widetilde{F}(h),h)_{\mathcal{H}}=0,
	$
		so that $\widetilde{F}(h)\in T_h\mathcal{M}$. Hence, the modified evolution
		\begin{align}\label{eqn-2}
			\begin{cases}
				\displaystyle\frac{du(t)}{dt}=\widetilde{F}(u(t)), & t>0,\\
				u(0)=u_0,
			\end{cases}
		\end{align}
		is tangent to $\mathcal{M}$ and therefore preserves the constraint
		$\|u(t)\|_{\mathcal{H}}=1$ whenever $u_0\in\mathcal{M}$.
		
		If $D(F)=\mathcal{H}$ and $F$ is locally Lipschitz continuous, then $\widetilde{F}$ is also locally Lipschitz continuous, and \eqref{eqn-2} admits a unique local-in-time solution. Moreover, for $u_0\in\mathcal{M}$, the solution remains on $\mathcal{M}$ throughout its interval of existence.
		
		\subsubsection{One-dimensional constrained KS equation}
	We first consider the one-dimensional problem \eqref{eq:1d-v-KS equation}. 	To this end, denote  the nonlinear mapping $F: H^4(\mathbb T_L)\to L^2(\mathbb T_L)$ by 
	$$F(v)=-v_{xxxx}-v_{xx}-vv_x.$$ 
	Within this framework, the modified (projected) vector field $\widetilde{F}(v)$, obtained by projecting $F(v)$ onto the tangent space of the unit sphere at $v$, is expressed as:
	\begin{align}\label{eqn-2-1}
		\widetilde{F}(v)& = \pi_v[F(v)] = -v_{xxxx}-v_{xx}-vv_x - (v,-v_{xxxx}-v_{xx}-vv_x)_{L^2} v 
		\nonumber\\&=-v_{xxxx}-v_{xx}-vv_x +(\|v_{xx}\|_{L^2}^2-\|v_{x}\|_{L^2}^2)v,
	\end{align}
	since $(v,vv_x)=0$ by integration by parts. 
	Thus the initial value problem  \eqref{eq:1d-v-KS equation} reduces to 
		\begin{equation}\label{eq:constrained-v}
		\left\{
		\begin{aligned}
			v_t+v_{xxxx}+v_{xx}+vv_x&=\lambda_v(t)v,
			&& (x,t)\in\mathbb T_L\times(0,T),\\
			v(x,0)&=v_0(x),
			&&x\in\mathbb T_L,
		\end{aligned}
		\right.
	\end{equation}
	where
	\begin{equation}\label{eq:lambda-v}
		\lambda_v(t)
		=
		\|v_{xx}(t)\|_{L^2}^2-\|v_x(t)\|_{L^2}^2, \ t\in(0,T)
	\end{equation}
	and 
	\begin{equation}
		v_0\in\mathcal{M}=\left\{v\in L^2(\mathbb T_L):\|v\|_{L^2}=1\right\}.
	\end{equation}

It is well known that the one-dimensional KS equation admits global solutions. Motivated by this fact, we establish the global solvability of the constrained one-dimensional KS equation \eqref{eq:constrained-v}. To this end, we employ a Fourier--Galerkin approximation. We first establish the local existence and uniqueness of solutions to the resulting finite-dimensional Galerkin system. We then derive uniform a priori energy estimates, with Agmon's inequality playing a crucial role in controlling the nonlinear terms and obtaining bounds that are independent of the Galerkin approximation. By the Banach--Alaoglu theorem, we extract weakly convergent subsequences from the sequence of approximate solutions. The Aubin--Lions--Simons compactness lemma then yields strong convergence in the appropriate function spaces, allowing us to pass to the limit in the Galerkin formulation and obtain a global solution. Finally, we verify that both the Galerkin approximations and their limit remain on the manifold \(\mathcal{M}\). This preservation of the constraint, together with the established regularity, leads to the uniqueness of the resulting strong solution.

	\subsubsection{Two- and three-dimensional constrained KS equation}
	Let us now consider the multidimensional problem \eqref{eq:vector-KS equation} and define the nonlinear mapping $F: H^4(\mathbb T_L^d;\mathbb R^d)\to L^2(\mathbb T_L^d;\mathbb R^d),$  $d\in\{2,3\}$ by  
	\begin{align*}
	F(\boldsymbol{v})=	-\Delta^2\boldsymbol{v}-\Delta \boldsymbol{v}-(\boldsymbol{v}\cdot\nabla)\boldsymbol{v}.
	\end{align*}
	To ensure that the dynamics remain tangent to the unit sphere, we replace \(F(\boldsymbol{v})\) by its orthogonal projection onto the tangent space at \(\boldsymbol{v}\), denoted by \(\widetilde{F}(\boldsymbol{v})\), which is given by
	\begin{align}
		\widetilde{F}(\boldsymbol{v})& = \pi_{\boldsymbol{v}}[F(\boldsymbol{v})]=-\Delta^2\boldsymbol{v}-\Delta \boldsymbol{v}-(\boldsymbol{v}\cdot\nabla)\boldsymbol{v}-(\boldsymbol{v},-\Delta^2\boldsymbol{v}-\Delta \boldsymbol{v}-(\boldsymbol{v}\cdot\nabla)\boldsymbol{v})_{L^2}\boldsymbol{v}
		\nonumber\\&=-\Delta^2\boldsymbol{v}-\Delta \boldsymbol{v}-(\boldsymbol{v}\cdot\nabla)\boldsymbol{v}+\left(\|\Delta\boldsymbol{v}\|_{L^2}^2-\|\nabla\boldsymbol{v}\|_{L^2}^2-\frac12
		\int_{\mathbb T_L^d}
		|\boldsymbol{v}|^2\operatorname{div}\boldsymbol{v} dx\right)\boldsymbol{v},
	\end{align}
	where we have performed integration by parts. Thus, the initial value problem \eqref{eq:vector-KS equation} can be reformulated as
		\begin{equation}\label{eq:KSE}
		\begin{cases}
			\boldsymbol{v}_t+\Delta^2\boldsymbol{v}+\Delta \boldsymbol{v}+(\boldsymbol{v}\cdot\nabla)\boldsymbol{v}=\lambda_{\boldsymbol{v}}(t)\boldsymbol{v},
			& (x,t)\in\mathbb T_L^d\times(0,T),\\
			\boldsymbol{v}(0)=\boldsymbol{v}_0,\ x\in\mathbb T_L^d,
		\end{cases}
	\end{equation}
	where
	\begin{equation}\label{eq:lambda-1}
		\lambda_{\boldsymbol{v}}(t)
		=
		\norm{\Delta \boldsymbol{v}(t)} _{L^2}^2-\norm{\nabla \boldsymbol{v}(t)} _{L^2}^2
		-
		\frac12
		\int_{\mathbb T_L^d}
		|\boldsymbol{v}(t)|^2\operatorname{div}\boldsymbol{v}(t) dx
	\end{equation}
	and 
		\begin{equation}
			\boldsymbol{v}_0\in\mathcal{M}=\left\{\boldsymbol{v}\in L^2(\mathbb T_L^d;\mathbb R^d):\|\boldsymbol{v}\|_{L^2}=1\right\}.
		\end{equation}

	The global well-posedness of the unconstrained KS equation \eqref{eq:vector-KS equation} in two and three spatial dimensions remains an open problem. We establish the existence and uniqueness of a maximal strong solution to the constrained two- and three-dimensional KS equation \eqref{eq:KSE}, in the sense of Definition \ref{def-strong-m}, up to a maximal time \(T_{\max}\). The local well-posedness is obtained by combining suitable semigroup estimates with the Banach contraction principle.
	
		In the \emph{two-dimensional case}, we further prove that the maximal solution extends globally in time. A crucial feature of the constrained problem is that the \(L^2\)-norm is preserved along the evolution. Indeed, taking the \(L^2\)-inner product of the equation with \(\boldsymbol{v}\) and using the constraint yields the conservation of the \(L^2\)-energy, without requiring an estimate of the nonlinear term
	$ \bigl((\boldsymbol{v}\cdot\nabla)\boldsymbol{v},\boldsymbol{v}\bigr). $
	However, to obtain higher-order a priori estimates, we take the inner product of the equation with \(\boldsymbol{v}_t\), which leads to the nonlinear term
	$ \bigl((\boldsymbol{v}\cdot\nabla)\boldsymbol{v},\boldsymbol{v}_t\bigr). $
	We carefully estimate this term by exploiting Agmon's inequality together with the \(L^2\)-constraint. These estimates provide the necessary control of the higher-order norms and rule out finite-time blow-up. Consequently, the maximal solution in two dimensions can be extended globally in time, yielding the global well-posedness of the constrained two-dimensional KS equation.
	
	The main novelties of this work can be summarized as follows:
	\begin{itemize}
		\item Global well-posedness of the \(L^2\)-constrained one-dimensional KS equation \eqref{eq:constrained-v} (Theorem \ref{thm-one-d-global}).
	\item 	Local well-posedness of the \(L^2\)-constrained two- and three-dimensional KS equations \eqref{eq:KSE} (Theorem \ref{thm-local}).
	\item 	Global well-posedness of the \(L^2\)-constrained two-dimensional KS equation (Theorem \ref{thm-main}).
	\end{itemize}
	The third contribution provides a partial resolution, within the constrained setting, of the well-known open problem of global well-posedness for the two-dimensional KS equation. We point out that the additional Lagrange multiplier appearing in	\eqref{eq:KSE} introduces further difficulties in establishing global well-posedness. These difficulties can be handled in two spatial dimensions, but remain unresolved in three dimensions, particularly in the presence of the nonlinear term $(\boldsymbol{v}\cdot\nabla)\boldsymbol{v}$ (Remark \ref{rem-3D}).
	
	\subsection{Main results}
We present the main results of this work in this section. We begin by introducing the notion of a strong solution to problem \eqref{eq:constrained-v}.

	\begin{definition}\label{def-strong}
		Let $T>0$ and let
		\begin{equation*}
			v_0\in H^4(\mathbb T_L)\cap\mathcal{M}.
		\end{equation*}
		Then a function
		\begin{equation*}
			v \in
			C([0,\infty);H^2(\mathbb T_L)\cap\mathcal{M})
			\cap
			H^1_{\mathrm{loc}}
			(0,\infty;L^2(\mathbb T_L))\cap L^2_{\mathrm{loc}}
			(0,\infty;H^4(\mathbb T_L))
		\end{equation*}
		is called \emph{strong solution} if it satisfies the following:
		\begin{itemize}
			\item $v_t\in L^2_{\mathrm{loc}}(0,\infty;L^2(\mathbb T_L))$ and	$\mathcal{N}(v)\in L^2_{\mathrm{loc}}(0,\infty;L^2(\mathbb T_L)),$
			where
			\begin{equation*}
				\mathcal{N}(v)=
				-v_{xxxx}-v_{xx}-vv_x +(\|v_{xx}\|_{L^2}^2-\|v_{x}\|_{L^2}^2) v.
			\end{equation*}
			\item $v_t=\mathcal{N}(v) \ \text{ in }\   L^2(\mathbb T_L)\ \text{ for a.e. }\ t\in (0,T).$\vspace{0.2cm}
			\item The initial condition is prescribed by
			\begin{equation*}
				v(0)=v_0\ \text{ in }\ H^2(\mathbb T_L). 
			\end{equation*}
		\end{itemize}
	\end{definition}

	The following result establishes the existence and uniqueness of a strong solution to the \(L^2\)-constrained one-dimensional KS equation \eqref{eq:constrained-v}.

	\begin{theorem}\label{thm-one-d-global}
		Let $\mathbb{T}_L=\mathbb{R}/(L\mathbb{Z})$ and let
		\begin{align*}
		v_0\in H^2(\mathbb{T}_L)\cap\mathcal{M}.
		\end{align*}
		%		\begin{equation}\label{eq:constrained-KS}
			%			\left\{
			%			\begin{aligned}
				%				u_t+u_{xxxx}+u_{xx}+uu_x&=\lambda(t)u,\ \ 
				%				 (x,t)\in\mathbb{T}_L\times(0,\infty),\\
				%				u(x,0)&=u_0(x),\ \ 
				%				x\in\mathbb{T}_L,
				%			\end{aligned}
			%			\right. 
			%		\end{equation}
		%		where
		%		\begin{equation}\label{eq:lambda}
			%			\lambda(t)
			%			=
			%			\|u_{xx}(t)\|_{L^2}^2-	\|u_x(t)\|_{L^2}^2. 
			%		\end{equation}
		Then $L^2$-constrained KS equation  \eqref{eq:constrained-v}  admits a unique global strong solution
		\begin{align*}
		v\in
		C([0,\infty);H^2(\mathbb{T}_L))
		\cap
		H^1_{\mathrm{loc}}(0,\infty;L^2(\mathbb{T}_L))\cap L^2_{\mathrm{loc}}(0,\infty;H^4(\mathbb{T}_L)). 
		\end{align*}
		Moreover, the following invariant condition is satisfied: 
		\begin{align*}
		\|v(t)\|_{L^2(\mathbb{T}_L)}=1
		\ \ \text{ for all }\ t\geq0.
		\end{align*}
	\end{theorem}
	
	The proof of Theorem \ref{thm-one-d-global} is given in Section \ref{sec3}.
		
		Let us now provide the definition of a strong solution to problem \eqref{eq:KSE}.

	\begin{definition}\label{def-strong-m}
		Let $T>0$ and let
		\begin{equation*}
			\boldsymbol{v}_0\in H^2(\mathbb T_L^d;\mathbb R^d)\cap\mathcal{M}.
		\end{equation*}
		Then a function
		\begin{equation*}
			\boldsymbol{v} \in
			C([0,\infty);H^2(\mathbb T_L^d;\mathbb R^d))
			\cap
			H^1_{\mathrm{loc}}
			(0,\infty;L^2(\mathbb T_L^d;\mathbb R^d))\cap L^2_{\mathrm{loc}}
			(0,\infty;H^4(\mathbb T_L^d;\mathbb R^d))
		\end{equation*}
		is called \emph{strong solution} if it satisfies the following:
		\begin{itemize}
			\item $\boldsymbol{v}_t\in L^2_{\mathrm{loc}}(0,\infty;L^2(\mathbb T_L^d;\mathbb R^d))$ and	$\mathcal{N}(\boldsymbol{v})\in L^2_{\mathrm{loc}}(0,\infty;L^2(\mathbb T_L^d;\mathbb R^d)),$
			where
			\begin{equation*}
				\mathcal{N}(\boldsymbol{v})=-\Delta^2\boldsymbol{v}-\Delta \boldsymbol{v}-(\boldsymbol{v}\cdot\nabla)\boldsymbol{v}+\left(\|\Delta\boldsymbol{v}\|_{L^2}^2-\|\nabla\boldsymbol{v}\|_{L^2}^2-\frac12
				\int_{\mathbb T_L^d}
				|\boldsymbol{v}|^2\operatorname{div}\boldsymbol{v} dx\right)\boldsymbol{v}.
			\end{equation*}
			\item $\boldsymbol{v}_t=\mathcal{N}(\boldsymbol{v}) \ \text{ in }\   L^2(\mathbb T_L^d;\mathbb R^d)\ \text{ for a.e. }\ t\in (0,T).$\vspace{0.2cm}
			\item The initial condition is prescribed by
			\begin{equation*}
				\boldsymbol{v}(0)=\boldsymbol{v}_0\ \text{ in }\ H^2(\mathbb T_L^d;\mathbb R^d). 
			\end{equation*}
		\end{itemize}
	\end{definition}
	
	The following result guarantees the local existence and uniqueness of a strong solution to the \(L^2\)-constrained two- and three-dimensional KS equation \eqref{eq:KSE}.

		\begin{theorem}[Local well-posedness in dimensions $d=2,3$]\label{thm-local}
		Let $d\in\{2,3\}$ and let
		$
		\mathbb T_L^d=(\mathbb R/L\mathbb Z)^d.
		$
		For any
		$
		\boldsymbol{v}_0\in
		H^2(\mathbb T_L^d;\mathbb R^d)\cap\mathcal{M}.
		$
%		consider the initial value problem
%		\begin{equation}\label{eqn-equation-main}
%			\begin{cases}
%				\boldsymbol{v}_t+\Delta^2\boldsymbol{v}
%				+\Delta\boldsymbol{v}
%				+(\boldsymbol{v}\cdot\nabla)\boldsymbol{v}
%				=
%				\lambda_{\boldsymbol v}(t)\boldsymbol{v},
%				& (x,t)\in\mathbb T_L^d\times(0,T),\\begin{align*}1mm]
%				\boldsymbol{v}(x,0)=\boldsymbol{v}_0(x),
%				& x\in\mathbb T_L^d,
%			\end{cases}
%		\end{equation}
%		where $\lambda_{\boldsymbol v}$ is defined in \eqref{eqn-lambda-1}. 
%		
			Then there exists a time
		$
		T=T\bigl(\|\boldsymbol v_0\|_{H^2},L\bigr)>0
		$
		such that \eqref{eq:KSE} admits a unique solution satisfying
		\begin{align}\label{eqn-reg-1}
			\boldsymbol v\in
			C\bigl([0,T];H^2(\mathbb T_L^d;\mathbb R^d)\cap\mathcal{M}\bigr)
			\cap
			H^1\bigl(0,T;L^2(\mathbb T_L^d;\mathbb R^d)\bigr)
			\cap
			L^2\bigl(0,T;H^4(\mathbb T_L^d;\mathbb R^d)\bigr).
		\end{align}
		
		Moreover, for every $R>0$, there exists
		$
		T_R=T_R(L)>0
		$
		such that, whenever
		\begin{align*}
		\|\boldsymbol v_0\|_{H^2}
		+
		\|\boldsymbol w_0\|_{H^2}
		\le R,
		\end{align*}
		the corresponding solutions $\boldsymbol v$ and $\boldsymbol w$ exist on
		$[0,T_R]$ and satisfy
		\begin{equation}\label{eq:local-Lipschitz}
			\sup_{0\le t\le T_R}
			\|\boldsymbol v(t)-\boldsymbol w(t)\|_{H^2}
			\le
			C_R
			\|\boldsymbol v_0-\boldsymbol w_0\|_{H^2}.
		\end{equation}
		In particular, the solution map
		$$
		H^2(\mathbb T_L^d;\mathbb R^d)\ni	\boldsymbol v_0
		\mapsto
		\boldsymbol v\in C([0,T_R];H^2(\mathbb T_L^d;\mathbb R^d))
		$$
		is locally Lipschitz .
		
		Finally, if $T_{\max}$ denotes the maximal existence time of the
		solution, then
		\begin{align}\label{eqn-max}
			T_{\max}<\infty
			\quad\Longrightarrow\quad
			\limsup_{t\uparrow T_{\max}}
			\|\boldsymbol v(t)\|_{H^2}
			=\infty.
		\end{align}
		Equivalently, the solution can be continued beyond any finite time
		$T_{\max}$ as long as
		\begin{align*}
		\sup_{0\le t<T_{\max}}
		\|\boldsymbol v(t)\|_{H^2}<\infty.
		\end{align*}
	\end{theorem}

		The proof of Theorem \ref{thm-local} is provided in Section \ref{sec-4}.  The following result demonstrates the global existence and uniqueness of a strong solution to the \(L^2\)-constrained two-dimensional KS equation \eqref{eq:KSE}.

			\begin{theorem}\label{thm-main}
			Let
			$
			\mathbb T_L^2=(\mathbb R/L\mathbb Z)^2
			$
			and let
			\begin{align*}
			\boldsymbol{v}_0\in H^2(\mathbb T_L^2;\mathbb R^2)\cap\mathcal{M}. 
			\end{align*}
			Consider the equation \eqref{eq:KSE} for $d=2$.
%			\begin{equation}\label{eq:KSE}
%				\begin{cases}
%					\boldsymbol{v}_t+\Delta^2\boldsymbol{v}+\Delta \boldsymbol{v}+(\boldsymbol{v}\cdot\nabla)\boldsymbol{v}=\lambda_{\boldsymbol{v}}(t)\boldsymbol{v},
%					& (x,t)\in\mathbb T_L^2\times(0,\infty),\\
%					\boldsymbol{v}(0)=\boldsymbol{v}_0,\ x\in\mathbb T_L^2,
%				\end{cases}
%			\end{equation}
%			where
%			\begin{align*}
%			(\boldsymbol{v}\cdot\nabla)\boldsymbol{v}
%			=
%			v_1\partial_1\boldsymbol{v}+v_2\partial _{2}\boldsymbol{v}
%			\end{align*}
%			and
%			\begin{equation}\label{eq:lambda-1}
%				\lambda_{\boldsymbol{v}}(t)
%				=
%				\norm{\Delta \boldsymbol{v}(t)} _{L^2}^2-\norm{\nabla \boldsymbol{v}(t)} _{L^2}^2
%				-
%				\frac12
%				\int_{\mathbb T_L^2}
%				|\boldsymbol{v}(t)|^2\operatorname{div}\boldsymbol{v}(t) dx.
%			\end{equation}
			Then \eqref{eq:KSE} admits a unique global solution
			\begin{align*}
			\boldsymbol{v}\in
			C([0,\infty);H^2(\mathbb T_L^2;\mathbb R^2))
			\cap
			H^1_{\mathrm{loc}}
			(0,\infty;L^2(\mathbb T_L^2;\mathbb R^2))\cap L^2_{\mathrm{loc}}
			(0,\infty;H^4(\mathbb T_L^2;\mathbb R^2)).
			\end{align*}
			Moreover, the following invariance result holds: 
			\begin{align*}
			\norm{\boldsymbol{v}(t)} _{L^2}=1
			\ \text{ for every }\ t\geq0.
			\end{align*}
		\end{theorem}

	The proof of Theorem \ref{thm-main} is deferred to Section \ref{sec-5}.

		\subsection{Organization of the paper} The remainder of the paper is organized as follows. In Section \ref{sec-2}, we	introduce the functional framework and collect the preliminary analytical	results needed throughout the paper. In particular, we define the linear	operator associated with the constrained KS equation and	show that it generates an analytic $C_0$-semigroup (Proposition \ref{prop-analytic}). We then establish the	relevant semigroup estimates and properties that will be used in the	subsequent analysis. We also prove the local Lipschitz continuity of the	nonlinear operator associated with the equation.
		
		Section \ref{sec3} is devoted to the one-dimensional constrained	KS equation. We first establish, in Proposition	\ref{prop-equiv-1}, the equivalence between the constrained slope formulation and the	corresponding scalar formulation. We then prove the well-posedness of the	one-dimensional constrained KS equation in Theorem	\ref{thm-one-d-global}. The proof is based on a Fourier--Galerkin approximation, followed	by uniform a priori estimates and a compactness argument to pass to the	limit.
		
		In Section \ref{sec-4}, we turn to the constrained KS	equation in two and three spatial dimensions. We begin by establishing the	equivalence between the constrained slope formulation and the scalar	formulation in the two-dimensional case. We then employ the Banach	contraction mapping theorem to establish the existence and uniqueness of a	local maximal solution to problem \eqref{eq:KSE} (Theorem \ref{thm-local}). 
		
		Finally, in Section \ref{sec-5}, by exploiting the geometric constraint together with Agmon's	inequality, we derive the crucial a priori estimates that prevent finite-time 	blow-up and allow the local maximal solution to be extended globally. This	yields the global well-posedness of the constrained KS	equation in two  spatial dimensions (Theorem \ref{thm-main}).

	\section{Preliminaries}\label{sec-2}
	This section introduces the notation and fundamental concepts that will be employed throughout the subsequent analysis.
	
	\subsection{Functional setting}	In what follows, let
	\begin{align*}
	\mathbb T_L^d:=\left(\mathbb R/L\mathbb Z\right)^d,
	\ \  d\in\{1,2,3\},
	\end{align*}
	denote the $d$-dimensional torus of period $L>0$. We identify
	$\mathbb T_L^d$ with the periodic box $[0,L]^d$, with opposite faces
	identified. Let
$
	x=(x_1,\ldots,x_d)\in\mathbb T_L^d
$
	denote the spatial variable. We consider the vector field
	\begin{align*}
	\boldsymbol{v}=\nabla\phi
	=
	\left(
	\frac{\partial\phi}{\partial x_1},
	\ldots,
	\frac{\partial\phi}{\partial x_d}
	\right)^{\!\top},
	\end{align*}
	and hence
	$
	\boldsymbol{v}:\mathbb T_L^d\to\mathbb R^d.
	$
	In particular, for $d=1$, $\boldsymbol{v}=v_1$ is scalar-valued; for $d=2$,
$
	\boldsymbol{v}=(v_1,v_2)^\top;
$
	and for $d=3$,
	$
	\boldsymbol{v}=(v_1,v_2,v_3)^\top.
$
We denote by
$
L^2(\mathbb T_L^d;\mathbb R^d),
$
the space of (equivalence classes of) strongly measurable functions
$
\boldsymbol{u}:\mathbb T_L^d\to\mathbb R^d
$
such that
\begin{align*}
\int_{\mathbb T_L^d}|\boldsymbol{u}(x)|_{\mathbb R^d}^2 dx<\infty.
\end{align*}
Here, two functions $\boldsymbol{u},\boldsymbol{v}:\mathbb T_L^d\to\mathbb R^d$ are identified
whenever
$
\boldsymbol{u}(x)=\boldsymbol{v}(x)
\ \text{ for almost every }\ x\in\mathbb T_L^d.
$
For
$
\boldsymbol{u}=(u_1,\ldots,u_d)^\top,\ 
\boldsymbol{v}=(v_1,\ldots,v_d)^\top
\in L^2(\mathbb T_L^d;\mathbb R^d),
$ the space $L^2(\mathbb T_L^d;\mathbb R^d)$ is equipped with the inner
product
\begin{align*}
(\boldsymbol{u},\boldsymbol{v})_{L^2}
:=
\int_{\mathbb T_L^d}\boldsymbol{u}(x)\cdot \boldsymbol{v}(x) dx
=
\sum_{i=1}^d
\int_{\mathbb T_L^d}u_i(x)v_i(x) dx,
\end{align*}
and the corresponding norm
\begin{align*}
\|\boldsymbol{u}\|_{L^2}
:=
\left(
\int_{\mathbb T_L^d}|\boldsymbol{u}(x)|_{\mathbb R^d}^2 dx
\right)^{1/2}.
\end{align*} 
		Every $\boldsymbol{u}\in L^2(\mathbb T_L^d;\mathbb R^d)$ admits the Fourier
	representation 
	\begin{align*}
	\boldsymbol{u}(x)
	=
	\sum_{k\in\mathbb Z^d}
	\widehat{\boldsymbol{u}}_k
	e^{\,i\frac{2\pi}{L}k\cdot x},
	\ \ 
	\widehat{\boldsymbol{u}}_k\in\mathbb C^d,\  \text{ where }\ 
	\widehat{\boldsymbol{u}}_{-k}=\overline{\widehat{\boldsymbol{u}}}_k.
	\end{align*}
The conjugate-symmetry condition ensures that the Fourier series is
real-valued. The Parseval identity gives
	$$\|\boldsymbol{u}\|_{L^2}^2=L^d\sum_{k\in\mathbb{Z}^d}|\widehat{\boldsymbol{u}}_k|_{\mathbb{C}^d}^2.$$ 
	The Fourier modes are
$
	e_k(x)
	=
	e^{\,i\frac{2\pi}{L}k\cdot x},
	\ k\in\mathbb Z^d,
$
	and satisfy
	\begin{align*}
	-\Delta e_k
	=
	\mu_k e_k,
	\ \
	\mu_k
	=
	\frac{4\pi^2}{L^2}|k|^2.
	\end{align*}
The normalized coefficients associated with the orthonormal basis are $\widetilde{e}_k(x)=\frac{1}{L^{d/2}} e^{\,i\frac{2\pi}{L}k\cdot x},
\ k\in\mathbb Z^d.$	For $s\in\mathbb R$, we define the Sobolev space
$
	H^s(\mathbb T_L^d;\mathbb R^d)
$
	by
	\begin{align*}
	\|\boldsymbol{u}\|_{H^s(\mathbb T_L^d;\mathbb R^d)}^2=\|\boldsymbol{u}\|_{H^s}^2:
	=L^d
	\sum_{k\in\mathbb Z^d}
	\left(
	1+\frac{4\pi^2}{L^2}|k|^2
	\right)^s
	|\widehat{\boldsymbol{u}}_k|_{\mathbb C^d}^2.
	\end{align*}
	In particular,
	\begin{align*}
	H^1(\mathbb T_L^d;\mathbb R^d)
	=
	\left\{
	\boldsymbol{u}\in L^2(\mathbb T_L^d;\mathbb R^d):
	\nabla \boldsymbol{u}\in L^2(\mathbb T_L^d;\mathbb R^{d\times d})
	\right\}.
	\end{align*}
	We denote $H^s(\mathbb T_L;\mathbb R)$ as $H^s(\mathbb T_L)$, for all $s\geq 0$. We use the following periodic elliptic estimate in the sequel:
	\begin{equation}\label{eq:elliptic}
		\norm{\boldsymbol{u}}_{H^4}
		\leq
		C_L
		\left(
	\norm{\boldsymbol{u}} _{L^2}+	\norm{\Delta^2\boldsymbol{u}} _{L^2}
		\right).
	\end{equation}
	For completeness, we provide a proof. 
	Indeed, using the Fourier representation
	$
	\boldsymbol{u}(x)
	=
	\sum_{k\in\mathbb Z^d}
	\widehat{\boldsymbol{u}}_k e^{i\xi_k\cdot x},
	\ 
	\xi_k=\frac{2\pi}{L}k,
	$
	we have
	\begin{align*}
		\|\boldsymbol{u}\|_{H^4}^2
		&=
		L^d\sum_{k\in\mathbb Z^d}
		(1+|\xi_k|^2)^4
		|\widehat{\boldsymbol{u}}_k|_{\mathbb C^d}^2\leq
		8L^d\sum_{k\in\mathbb Z^d}
		(1+|\xi_k|^8)
		|\widehat{\boldsymbol{u}}_k|_{\mathbb C^d}^2,
	\end{align*}
	where we have used the elementary inequality
	$
	(1+r^2)^4\leq 8(1+r^8),
	\  r\geq0.
	$
	On the other hand,
	\[
	\Delta^2\boldsymbol{u}
	=
	\sum_{k\in\mathbb Z^d}
	|\xi_k|^4\widehat{\boldsymbol{u}}_k
	e^{i\xi_k\cdot x},
	\]
	and hence, by Parseval's identity,
	\begin{align*}
		L^d\sum_{k\in\mathbb Z^d}
		|\widehat{\boldsymbol{u}}_k|_{\mathbb C^d}^2
		=
		\|\boldsymbol{u}\|_{L^2}^2 \ \text{ and }\ 
		L^d\sum_{k\in\mathbb Z^d}
		|\xi_k|^8
		|\widehat{\boldsymbol{u}}_k|_{\mathbb C^d}^2
		=
		\|\Delta^2\boldsymbol{u}\|_{L^2}^2.
	\end{align*}
	Therefore, we infer 
	\[
	\|\boldsymbol{u}\|_{H^4}^2
	\leq
	C_L
	\left(
	\|\boldsymbol{u}\|_{L^2}^2
	+
	\|\Delta^2\boldsymbol{u}\|_{L^2}^2
	\right),
	\]
	which, upon taking the square root, yields \eqref{eq:elliptic}.
	
	Although \(\boldsymbol v=\nabla\phi\) is automatically a zero-mean vector field on \(\mathbb T_L^d\), we do not impose the zero-mean condition explicitly in the definition of the phase space. Indeed, none of the estimates developed below relies on the Poincaré inequality or requires the zero Fourier mode to be excluded. We therefore work in the full space
	$ H^s(\mathbb T_L^d;\mathbb R^d), $
		while noting that the particular solutions arising from the representation \(\boldsymbol v=\nabla\phi\) belong to the corresponding zero-mean subspace. This choice enables us, whenever possible, to establish our results without imposing the structural assumption \(\boldsymbol v=\nabla\phi\).
	Nevertheless, the results obtained for the vector formulation can be directly transferred to the scalar formulation by identifying \(\boldsymbol v\) with \(\nabla\phi\). It is important to emphasize, however, that in the absence of the gradient-field assumption \(\boldsymbol v=\nabla\phi\), the vector and scalar formulations of the KS equation are not necessarily equivalent. This distinction is addressed in Propositions \ref{prop-equiv-1} and \ref{prop-equivalence-2}.

\subsection{The linear operator and its analytic semigroup}
	Let
	$
	\mathbb{X}:=L^2(\mathbb T_L^d;\mathbb R^d),
	\  d\in\{2,3\},
	$
	and define the linear operator
	\begin{align*}
		A:=\Delta^2+\Delta\ \text{
			with domain }\ 
		D(A):=H^4(\mathbb T_L^d;\mathbb R^d).
	\end{align*}
	Then, we have the following result:
	\begin{proposition}\label{prop-analytic}
	The operator	$-A$ also generates an analytic $C_0$-semigroup
		$
		\{S_A(t)\}_{t\geq 0}
		$
		such that 
		\begin{align}\label{eqn-est-3}
			\|S_A(t)\|_{\mathcal L(\mathbb{X})}
			\leq e^{\ell t},
			\  t\geq0,
		\end{align}
		for some 	$
		\ell>\frac14
		$. 
		\end{proposition}
		\begin{proof}
	For $k\in\mathbb Z^d$, define
	$
	\xi_k:=\frac{2\pi}{L}k
	$
	and
	$
	e_k(x):=e^{i\xi_k\cdot x}.
	$
	Then
	$
	\Delta e_k=-|\xi_k|^2e_k,
	$
	and consequently
	\begin{align}\label{eqn-est-0}
		Ae_k
		=
		\bigl(|\xi_k|^4-|\xi_k|^2\bigr)e_k.
	\end{align}
	Writing
	$
	r_k:=|\xi_k|^2,
	$
	we have
	\begin{align*}
	r_k^2-r_k
	=
	\left(r_k-\frac12\right)^2-\frac14
	\geq-\frac14.
	\end{align*}
	Let us fix a constant
	$
	\ell>\frac14
	$
	and define
	\begin{align}\label{eqn-sel-b}
		B:=A+\ell I.
	\end{align}
	Then $B$ is self-adjoint on $\mathbb{X}$. Its Fourier symbol is
	$
	\beta_k
	:=
	|\xi_k|^4-|\xi_k|^2+\ell,
	$
	and 
	\begin{align}\label{eqn-betak}
		\beta_k
		:=
		\left(|\xi_k|^2-\frac12\right)^2
		+\ell-\frac14 \geq \ell-\frac14>0,
		\ \ k\in\mathbb Z^d,
	\end{align}
	so that $B$ is strictly positive. We now construct the semigroup generated by $-B$. For $t\geq0$, let us define
	\begin{align*}
		S_B(t)\boldsymbol{u}
		:=
		\sum_{k\in\mathbb Z^d}
		e^{-t\beta_k}\widehat{\boldsymbol{u}}_ke_k,
		\ \text{ where }\ 
		\boldsymbol{u}(x)=\sum_{k\in\mathbb Z^d}\widehat{\boldsymbol{u}}_ke_k(x),
		\ 
		\widehat{\boldsymbol{u}}_k\in\mathbb C^d.
	\end{align*}
	It can be easily seen that $S_B(t+s)=S_B(t)S_B(s),
	\ 
	S_B(0)=I.$  	It remains to verify strong continuity at $t=0$. By Parseval's identity, we have
	\begin{align*}
		\|S_B(t)	\boldsymbol{u}-	\boldsymbol{u}\| _{L^2}^2
		&=
		L^d\sum_{k\in\mathbb Z^d}
		|e^{-t\beta_k}-1|^2|	\widehat{\boldsymbol{u}}_k|_{\mathbb{C}^d}^2.
	\end{align*}
	For every fixed $k$,
	$
	e^{-t\beta_k}\to1
	\ \text{ as }\ t\downarrow0,
	$
	while
	$
	|e^{-t\beta_k}-1|^2\leq4.
	$
	Since
	$L^d\sum_{k\in\mathbb Z^d}|	\widehat{\boldsymbol{u}}_k|_{\mathbb{C}^d}^2=
	\|	\boldsymbol{u}\| _{L^2}^2<\infty,$
	the dominated convergence theorem gives
	\begin{align*}
	\lim_{t\downarrow0}\|S_B(t)	\boldsymbol{u}-	\boldsymbol{u}\| _{L^2}=0.
	\end{align*}
	Since $\beta_k>0$,
	$
	0<e^{-t\beta_k}\leq1.
	$
	Therefore, by Parseval's identity, we have 
	\begin{align*}
		\|S_B(t)	\boldsymbol{u}\|_{L^2}^2
		=
		L^d\sum_{k\in\mathbb Z^d}
		e^{-2t\beta_k}|	\widehat{\boldsymbol{u}}_k|_{\mathbb{C}^d}^2\leq
		L^d\sum_{k\in\mathbb Z^d}|	\widehat{\boldsymbol{u}}_k|_{\mathbb{C}^d}^2=
		\|\boldsymbol{u}\|_{L^2}^2.
	\end{align*}
	Hence, we conclude 
	\begin{align}\label{eqn-101}
	\|S_B(t)\|_{\mathcal L(\mathbb{X})}\leq1,
	\ t\geq0,
	\end{align}
	and  the semigroup is bounded, in fact contractive.
	Thus $\{S_B(t)\}_{t\geq0}$ is a contractive $C_0$-semigroup on $\mathbb{X}$.
	We next identify its generator. The domain of $B$ is
	\begin{align*}
	D(B)
	=
	\left\{
	\boldsymbol{u}=\sum_{k\in\mathbb Z^d}\widehat{\boldsymbol{u}}_ke_k\in \mathbb{X}:
	\sum_{k\in\mathbb Z^d}\beta_k^2|\widehat{\boldsymbol{u}}_k|_{\mathbb{C}^d}^2<\infty
	\right\}.
	\end{align*}
	For $u\in D(B)$, we have 
	\begin{align*}
	\frac{S_B(t)\boldsymbol{u}-\boldsymbol{u}}{t}
	=
	\sum_{k\in\mathbb Z^d}
	\frac{e^{-t\beta_k}-1}{t}\widehat{\boldsymbol{u}}_ke_k.
	\end{align*}
	For every $k$, it is immediate that 
	$
	\frac{e^{-t\beta_k}-1}{t}
	\longrightarrow
	-\beta_k
	\ \text{ as }\ t\downarrow0.
	$
	Moreover, since $1-e^{-x}\leq x$ for $x\geq0$, we get 
	$
	\left|
	\frac{e^{-t\beta_k}-1}{t}
	\right|
	=
	\frac{1-e^{-t\beta_k}}{t}
	\leq
	\beta_k.
	$ Therefore, it follows that 
	\begin{align*}
	\left|
	\frac{e^{-t\beta_k}-1}{t}+\beta_k
	\right|^2
	|\widehat{\boldsymbol{u}}_k|_{\mathbb{C}^d}^2
	\leq
	4\beta_k^2|\widehat{\boldsymbol{u}}_k|_{\mathbb{C}^d}^2.
	\end{align*}
	The right-hand side is summable because $\boldsymbol{u}\in D(B)$. Therefore, another application of the dominated convergence theorem gives
	\begin{align*}
		\frac{S_B(t)\boldsymbol{u}-\boldsymbol{u}}{t}
		\longrightarrow
		-\sum_{k\in\mathbb Z^d}\beta_k\widehat{\boldsymbol{u}}_ke_k
		=
		-B\boldsymbol{u}
		\ \text{ in }\ \mathbb{X}.
	\end{align*}
	Thus the generator of $\{S_B(t)\}_{t\geq0}$ is $-B$.
	We finally verify analyticity. For $z\in\mathbb C$ with
	$
	\operatorname{Re}z>0,
	$
	define
	\begin{align*}
	S_B(z)\boldsymbol{u}
	:=
	\sum_{k\in\mathbb Z^d}
	e^{-z\beta_k}\widehat{\boldsymbol{u}}_ke_k.
	\end{align*}
	Since
	$
	|e^{-z\beta_k}|
	=
	e^{-(\operatorname{Re}z)\beta_k}
	\leq1,
	$
	we have
	$
	\|S_B(z)\boldsymbol{u}\|_{L^2}\leq\|\boldsymbol{u}\|_{L^2}.
	$
	Let
	$
	0<\theta<\frac{\pi}{2}
	$
	and consider the closed sector
	\begin{align*}
	\Sigma_\theta
	:=
	\left\{
	z\in\mathbb C:
	|\arg z|\leq\frac{\pi}{2}-\theta
	\right\}.
	\end{align*}
	For $z\in\Sigma_\theta$, we immediately have 
	$
	\operatorname{Re}z
	\geq
	|z|\sin\theta.
	$
	Hence, for every $m\in\mathbb N$, this leads to 
	\begin{align*}
	\beta_k^m|e^{-z\beta_k}|
	\leq
	\beta_k^m
	e^{-|z|\sin\theta\,\beta_k}.
	\end{align*}
	On every compact subset of $\Sigma_\theta\cap\{z:\operatorname{Re}z>0\}$,
	the right-hand side is uniformly bounded in $k$, since
	$
	\sup_{r\geq0}r^m e^{-ar}<\infty
	\  (a>0).
	$
	Consequently, the Fourier series defining $S_B(z)$ may be differentiated
	term by term on such compact subsets. In particular, we have 
	\begin{align*}
	\frac{d}{dz}S_B(z)\boldsymbol{u}
	=
	-\sum_{k\in\mathbb Z^d}
	\beta_ke^{-z\beta_k}\widehat{\boldsymbol{u}}_ke_k,
	\end{align*}
	and therefore
	$
	z\mapsto S_B(z)\boldsymbol{u}
	$
	is holomorphic in the right half-plane.
	Thus
	$-B$ generates a bounded analytic $C_0$-semigroup on $\mathbb{X}$.
	Since
	$
	B=A+\ell I,
	$
	we have
	$
	A=B-\ell I.
	$
	Therefore
	$
	e^{-tA}
	=
	e^{\ell t}e^{-tB}.
	$
	Consequently, $-A$ also generates an analytic $C_0$-semigroup, which we denote by
	$
	S_A(t):=e^{-tA}=e^{\ell t}S_B(t), \ t\geq 0.
	$
	Using the contractivity of $S_B(\cdot)$ (see \eqref{eqn-101}), we arrive at \eqref{eqn-est-3}, which completes the proof. 
	\end{proof}

\subsubsection{The space $H^2$ as the fractional domain} 	For the local theory provided in Section \ref{sec-4}, it is convenient to use the strictly positive	operator $B$.	Let us first show that $\beta_k$ is comparable to $1+|\xi_k|^4$. Since 
	$$
	\beta_k
	=
	|\xi_k|^4-|\xi_k|^2+\ell=
	\left(|\xi_k|^2-\frac12\right)^2+\ell-\frac14,
	$$
	we distinguish two cases.
	If $|\xi_k|^2\leq1$, then
	\begin{align*}
	\beta_k\geq \ell-\frac14
	\geq
	\frac{\ell-\frac14}{2}(1+|\xi_k|^4),
	\end{align*}
	because $1+|\xi_k|^4\leq2$.
	If $|\xi_k|^2\geq1$, then
	$
	\left(|\xi_k|^2-\frac12\right)^2
	\geq
	\frac14|\xi_k|^4,
	$
	and hence
	\begin{align*}
	\beta_k
	\geq
	\frac14|\xi_k|^4
	\geq
	\frac18(1+|\xi_k|^4).
	\end{align*}
	Therefore there exists a constant $\ell_1>0$, depending only on $\ell$, such that
	\begin{align*}
	\beta_k\geq \ell_1(1+|\xi_k|^4).
	\end{align*}
	For the upper bound, we use
	$
	|\xi_k|^2
	\leq
	\frac12|\xi_k|^4+\frac12
	$
	to obtain
	\begin{align*}
	\beta_k
	\leq
	|\xi_k|^4+\frac12|\xi_k|^4+\frac12+\ell
	=
	\frac32|\xi_k|^4+\left(\ell+\frac12\right).
	\end{align*}
	Thus there exists $\ell_2>0$, depending only on $\ell$, such that
	\begin{align*}
	\beta_k\leq \ell_2(1+|\xi_k|^4).
	\end{align*}
	Consequently,
	\begin{align*}
	\beta_k\asymp1+|\xi_k|^4,
	\ \  k\in\mathbb Z^d.
	\end{align*}
	In particular, 
	there exist constants $\ell_1,\ell_2>0$, independent of $k$, such that
	\begin{align}\label{eqn-est-1}
		\ell_1(1+|\xi_k|^4)
		\leq
		\beta_k
		\leq
		\ell_2(1+|\xi_k|^4),
		\ \  k\in\mathbb Z^d.
	\end{align}
	%		Here and below, $a_k\asymp b_k$ means that there exist constants
	%		$\ell_1,\ell_2>0$, independent of $k$, such that
	%		\begin{align*}
	%		\ell_1b_k\leq a_k\leq \ell_2b_k.
	%		\end{align*}
	If
	$
	\boldsymbol{u}(x)=\sum_{k\in\mathbb Z^d}\widehat{\boldsymbol{u}}_ke^{i\xi_k\cdot x},
	$
	then, by the spectral definition of fractional powers, we have 
	\begin{align*}
	\|B^{1/2}\boldsymbol{u}\| _{L^2}^2
	=
	L^d\sum_{k\in\mathbb Z^d}
	\beta_k|\widehat{\boldsymbol{u}}_k|_{\mathbb{C}^d}^2.
	\end{align*}
	Using \eqref{eqn-est-1}, we find 
	\begin{align*}
	\|B^{1/2}\boldsymbol{u}\| _{L^2}^2
	\asymp
	\sum_{k\in\mathbb Z^d}
	(1+|\xi_k|^4)|\widehat{\boldsymbol{u}}_k|_{\mathbb{C}^d}^2.
	\end{align*}
	On the other hand,
	\begin{align}\label{eqn-est-12}
	\|\boldsymbol{u}\|_{H^2}^2
	\asymp
	\sum_{k\in\mathbb Z^d}
	(1+|\xi_k|^2)^2|\widehat{\boldsymbol{u}}_k|_{\mathbb{C}^d}^2,
	\end{align}
	and
	$
	(1+|\xi_k|^2)^2
	=
	1+2|\xi_k|^2+|\xi_k|^4
	\asymp
	1+|\xi_k|^4.
	$
	Consequently,
	$
	D(B^{1/2})
	=
	H^2(\mathbb T_L^d;\mathbb R^d)
	$
	with equivalent norms.
	Thus $H^2(\mathbb T_L^d;\mathbb R^d)$ is the natural phase space for the fourth-order
	parabolic equation \eqref{eq:vector-KS equation}.

\subsubsection{Semigroup properties} 
We need the following properties of the semigroup $\{S_A(t)\}_{t\geq 0}$ in the following sequel.
\begin{proposition}
	The semigroup $\{S_A(t)\}_{t\geq 0}$ enjoys the following properties: 
	\begin{align}
		&	\|S_A(t)\boldsymbol{u}\|_{H^2}
		\leq
		C_Tt^{-1/2}\|\boldsymbol{u}\|_{L^2},
		\ \ 0<t\leq T, \  \boldsymbol{u} \in L^2(\mathbb T_L^d;\mathbb R^d), \label{eq:SA-L2-H2} \\
		&	\|S_A(t)\boldsymbol{u}\|_{H^2}
		\leq
		C_T\|\boldsymbol{u}\|_{H^2},
		\  0\leq t\leq T, \ \boldsymbol{u} \in H^2(\mathbb T_L^d;\mathbb R^d). \label{eq:SA-H2-H2}
	\end{align}
	\end{proposition}
	\begin{proof}
	We infer from \eqref{eqn-est-0} that
	\begin{align*}
	S_A(t)\boldsymbol{u}
	=
	\sum_{k\in\mathbb Z^d}
	e^{-t(r_k^2-r_k)}
	\widehat{\boldsymbol{u}}_k e_k.
	\end{align*}
	where $r_k=|\xi_k|^2$. 
	By the Fourier characterization of
	the periodic $H^2$-norm and Parseval's identity, for all $\boldsymbol{u} \in L^2(\mathbb T_L^d;\mathbb R^d)$, we have 
	\begin{equation}\label{eqn-est-2}
		\|S_A(t)\boldsymbol{u}\|_{H^2}^2
		\leq
		C
		\sum_{k\in\mathbb Z^d}
		(1+|\xi_k|^2)^2
		e^{-2t(r_k^2-r_k)}
		|\widehat{\boldsymbol{u}}_k|_{\mathbb{C}^d}^2=
		C
		\sum_{k\in\mathbb Z^d}
		(1+r_k)^2
		e^{-2t(r_k^2-r_k)}
		|\widehat{\boldsymbol{u}}_k|_{\mathbb{C}^d}^2.
	\end{equation}
	It is therefore sufficient to establish
	\begin{equation}\label{eq:multiplier-estimate}
		(1+r)^2e^{-2(r^2-r)t}
		\leq
		C_Tt^{-1},
		\ \ 
		r\geq0,\ \ 0<t\leq T.
	\end{equation}
	We consider two cases.
	
	\noindent
	\emph{Case 1: $0\leq r\leq2$.}
	Since
	$
	r^2-r
	=
	\left(r-\frac12\right)^2-\frac14
	\geq-\frac14,
	$
	we have
	\begin{align*}
	e^{-2(r^2-r)t}
	\leq e^{t/2}
	\leq e^{T/2}.
	\end{align*}
	Moreover,
	$
	(1+r)^2\leq9,
	$
	so that
	\begin{align*}
		(1+r)^2e^{-2(r^2-r)t}
		\leq
		9e^{T/2}\leq 	9Te^{T/2}t^{-1},
	\end{align*}
	since $0<t\leq T$.

	\medskip
	\noindent
	\emph{Case 2: $r\geq2$.}
	In this case,
	$
	r^2-r
	=
	r^2\left(1-\frac1r\right)
	\geq
	\frac12r^2.
	$
	Therefore, we infer 
	\begin{align*}
	e^{-2(r^2-r)t}
	\leq
	e^{-tr^2}.
	\end{align*}
	Also, since $r\geq2$,
	$
	1+r\leq\frac32r\leq2r,
	$
	and hence
	$
	(1+r)^2\leq4r^2.
	$
	Consequently, for $0<t\leq T$, we have 
	\begin{align*}
		(1+r)^2e^{-2(r^2-r)t}
		\leq
		4r^2e^{-tr^2}\leq Ct^{-1},
	\end{align*}
	since  $	\sup_{r\geq0}r^2e^{-r^2}<\infty$. 
	Combining the two cases gives \eqref{eq:multiplier-estimate}, and \eqref{eq:SA-L2-H2} follows from \eqref{eqn-est-2}.
	Moreover, one can deduce the estimate  \eqref{eq:SA-H2-H2}  from \eqref{eqn-est-3}. 
	\end{proof}

	\subsection{Local Lipschitz continuity}  Since the spatial dimension is either two or three, the periodic Sobolev embeddings
	give
	$
	H^2(\mathbb T_L^d;\mathbb{R}^d)
	\hookrightarrow
	L^\infty(\mathbb T_L^d;\mathbb{R}^d)
	$
	and
	$
	H^2(\mathbb T_L^d;\mathbb{R}^d)
	\hookrightarrow
	W^{1,4}(\mathbb T_L^d;\mathbb{R}^d).
	$ 	
	Therefore, we have 
	\begin{equation}\label{eq:sobolev-correct}
		\|\boldsymbol{v}\|_{L^{\infty}}
		+
		\|\boldsymbol{v}\|_{L^4}
		+
		\|\nabla \boldsymbol{v}\|_{L^4}
		\leq
		C_L\|\boldsymbol{v}\|_{H^2}.
	\end{equation}
	We  define $F:H^2(\mathbb T_L^d;\mathbb{R}^d)\to L^2(\mathbb T_L^d;\mathbb{R}^d)$ by 
	\begin{align*}
		F(\boldsymbol{v})=-(\boldsymbol{v}\cdot\nabla)\boldsymbol{v}-\lambda_{\boldsymbol{v}}\boldsymbol{v},
	\end{align*}
	where
	\begin{equation}\label{eqn-lambda-1}
		\lambda_{\boldsymbol v}
		=
		\|\Delta\boldsymbol v\|_{L^2}^2
		-
		\|\nabla\boldsymbol v\|_{L^2}^2
		-
		\frac12
		\int_{\mathbb T_L^d}
		|\boldsymbol v|^2
		\operatorname{div}\boldsymbol v dx.
	\end{equation}
	
	\begin{lemma}
	The operator 
		$	F:H^2(\mathbb T_L^d;\mathbb{R}^d)\to L^2(\mathbb T_L^d;\mathbb{R}^d)$
		is locally Lipschitz for $d=2,3$. 
	\end{lemma} 
\begin{proof}	 If
	$\norm {\boldsymbol{v}}_{H^2}+\norm{\boldsymbol{w}}_{H^2}\le R$, then an application of H\"older's inequality yields 
	\begin{align}
		\norm{(\boldsymbol{v}\cdot\nabla)\boldsymbol{v}-(\boldsymbol{w}\cdot\nabla)\boldsymbol{w}}_{L^2}
		&\le \norm{\boldsymbol{v}-\boldsymbol{w}}_{L^4}\norm{\nabla \boldsymbol{v}}_{L^4}
		+\norm{\boldsymbol{w}}_{L^\infty}\norm{\nabla(\boldsymbol{v}-\boldsymbol{w})}_{L^2}\le C_R\norm{\boldsymbol{v}-\boldsymbol{w}}_{H^2}.
	\end{align}
	Moreover, we have 
	\begin{align}\label{eqn-102}
		&	\left|\norm{\Delta \boldsymbol{v}}_{L^2}^2-\norm{\Delta \boldsymbol{w}}_{L^2}^2\right|+	\left|\norm{\nabla \boldsymbol{v}}_{L^2}^2-\norm{\nabla \boldsymbol{w}}_{L^2}^2\right|
		\nonumber\\&\leq \left(\norm{\Delta \boldsymbol{v}}_{L^2}+\norm{\Delta \boldsymbol{w}}_{L^2}\right)\norm{\Delta (\boldsymbol{v}-\boldsymbol{w})}_{L^2}+\left(\norm{\nabla \boldsymbol{v}}_{L^2}+\norm{\nabla \boldsymbol{w}}_{L^2}\right)\norm{\nabla (\boldsymbol{v}-\boldsymbol{w})}_{L^2}
		\nonumber\\&	\le C_R\norm{\boldsymbol{v}-\boldsymbol{w}}_{H^2}.
	\end{align}
	Let us denote 
	\begin{align*}
	I(\boldsymbol{v})=\int_{\mathbb T_L^d}|\boldsymbol{v}|^2\operatorname{div}\boldsymbol{v} dx.
	\end{align*}
	Then,	one can observe that 
	\begin{align}
		|I(\boldsymbol{v})-I(\boldsymbol{w})|&\leq \int_{\mathbb T_L^d}\left||\boldsymbol{v}|^2-|\boldsymbol{w}|^2\right||\operatorname{div}\boldsymbol{v}| dx + \int_{\mathbb T_L^d}|\boldsymbol{w}|^2|\operatorname{div}(\boldsymbol{v}-\boldsymbol{w})| dx\\
		&\le \sqrt{d}\norm{|\boldsymbol{v}|+|\boldsymbol{w}|}_{L^4}\norm{\boldsymbol{v}-\boldsymbol{w}}_{L^4}\norm{\nabla \boldsymbol{v}}_{L^2}+\sqrt{d}\norm{\boldsymbol{w}}_{L^\infty}^2\norm{\nabla(\boldsymbol{v}-\boldsymbol{w})}_{L^2}\\
		&\le C_R\norm{\boldsymbol{v}-\boldsymbol{w}}_{H^2},
	\end{align}
where we have used the fact that $\|\operatorname{div}\boldsymbol{v}\|_{L^2}\leq\sqrt{d}\|\nabla \boldsymbol{v}\|_{L^2}$.	Using the definition of $\lambda_{\boldsymbol{v}}$ from \eqref{eqn-lambda-1} and \eqref{eqn-102}, we immediately have 
	\begin{align*}
	|\lambda_{\boldsymbol{v}}-\lambda_{\boldsymbol{w}}|\le C_R\norm{\boldsymbol{v}-\boldsymbol{w}}_{H^2},
	\ \  |\lambda_{\boldsymbol{w}}|\le C_R,
	\end{align*}
	and consequently
	\begin{align}\label{eq:FLip-1}
		\norm{F(\boldsymbol{v})-F(\boldsymbol{w})}_{L^2}&\le  \norm{(\boldsymbol{v}\cdot\nabla)\boldsymbol{v}-(\boldsymbol{w}\cdot\nabla)\boldsymbol{w}}_{L^2}+	|\lambda_{\boldsymbol{v}}-\lambda_{\boldsymbol{w}}|\|	\boldsymbol{v}\|_{L^2} +|\lambda_{\boldsymbol{w}}|\|\boldsymbol{v}-\boldsymbol{w}\|_{L^2}\nonumber\\&\leq C_R\norm{\boldsymbol{v}-\boldsymbol{w}}_{H^2}.
	\end{align}
	In particular, we deduce 
	\begin{align}\label{eq:FLip-2}
		\norm{F(\boldsymbol{v})}_{L^2}\le C_R\norm{\boldsymbol{v}}_{H^2}
		\ \text{ for }\ \norm{\boldsymbol{v}}_{H^2}\le R,
	\end{align}
	which completes the proof.
	\end{proof}

		\section{Proof of Theorem \ref{thm-one-d-global}}\label{sec3}
	
The aim of this section is to establish the well-posedness of the one-dimensional KS equation by means of a Fourier--Galerkin approximation and compactness arguments. Before establishing the well-posedness result, we first establish the equivalence between the constrained slope formulation and the scalar formulation of the 1D KS equation \eqref{eq:constrained-v}.

\begin{proposition}[Equivalence between the constrained slope and scalar formulations]\label{prop-equiv-1}
	Let $\mathbb T_L=\mathbb R/L\mathbb Z$ and let
	$v_0\in L^2(\mathbb T_L)$ satisfy
$
	\|v_0\|_{L^2}=1.
$
	Suppose that $v$ is a sufficiently regular periodic solution of \eqref{eq:constrained-v}.
%	\begin{equation}\label{eq:constrained-v}
%		\left\{
%		\begin{aligned}
%			v_t+v_{xxxx}+v_{xx}+vv_x&=\lambda(t)v,
%			&& (x,t)\in\mathbb T_L\times(0,T),\\
%			v(x,0)&=v_0(x),
%			&&x\in\mathbb T_L,
%		\end{aligned}
%		\right.
%	\end{equation}
%	where
%	\begin{equation}\label{eq:lambda-v}
%		\lambda(t)
%		=
%		\|v_{xx}(t)\|_{L^2}^2-\|v_x(t)\|_{L^2}^2.
%	\end{equation}
	Then, writing $	v=u_x,$
	the corresponding scalar function $u$ satisfies, up to an additive
	function of time,
	\begin{equation}\label{eq:constrained-u}
		u_t+u_{xxxx}+u_{xx}
		+\frac12(u_x)^2
		=\lambda_{u_x}(t)u,
	\end{equation}
	together with the constraint
	\begin{align*}
	\|u_x(t)\|_{L^2}=1, \ \text{ for all }\ t\geq 0,
	\end{align*}
where $\lambda_{u_x}(t)=
\|u_{xxx}(t)\|_{L^2}^2
-\|u_{xx}(t)\|_{L^2}^2.$ 	

Conversely, if $u$ satisfies \eqref{eq:constrained-u} and
	$\|u_x(t)\|_{L^2}=1$ for all $t\geq 0$, then $v=u_x$ satisfies
	\eqref{eq:constrained-v}.
\end{proposition}

\begin{proof}
	We divide the proof into the following three parts.
	
	\smallskip
	\noindent
	\textbf{Step 1:} \emph{Preservation of the constraint.}
	Taking the $L^2$-inner product of
	\eqref{eq:constrained-v} with $v$, we obtain
	\begin{align}\label{eqn-111}
	(v_t,v)+(v_{xxxx},v)+(v_{xx},v)+(vv_x,v)
	=\lambda_v(t)\|v\|_{L^2}^2.
	\end{align}
	By periodicity and integration by parts, we find 
	\begin{align*}
	(v_t,v)=\frac12\frac{d}{dt}\|v\|_{L^2}^2,\ 
	(v_{xxxx},v)=\|v_{xx}\|_{L^2}^2,\ 
	(v_{xx},v)=-\|v_x\|_{L^2}^2.
	\end{align*}
	Moreover, we have 
	\begin{align*}
	(vv_x,v)
	=
	\int_{\mathbb T_L}v^2v_x dx
	=
	\frac13\int_{\mathbb T_L}(v^3)_x dx
	=0.
	\end{align*}
	Consequently, we infer from \eqref{eqn-111} that 
	\begin{equation}\label{eq:energy-v}
		\frac12\frac{d}{dt}\|v(t)\|_{L^2}^2
		+\|v_{xx}(t)\|_{L^2}^2
		-\|v_x(t)\|_{L^2}^2
		=
		\lambda_v(t)\|v(t)\|_{L^2}^2,
	\end{equation}
for all $t\in[0,T]$. 	Since
$	\lambda_v(t)
	=
	\|v_{xx}(t)\|_{L^2}^2
	-\|v_x(t)\|_{L^2}^2,$
	equation \eqref{eq:energy-v} becomes
	\begin{align*}
	\frac12\frac{d}{dt}\left(\|v(t)\|_{L^2}^2-1\right)
	=
	\lambda_v(t)\big(\|v(t)\|_{L^2}^2-1\big).
	\end{align*}
Therefore, if 
$	\|v_0\|_{L^2}=1,$
	then
$	\|v(t)\|_{L^2}^2\equiv1$
	is a solution of the above scalar differential equation. By
	uniqueness for this ordinary differential equation, we find 
$$\|v(t)\|_{L^2}=1,\ \  t\geq 0.$$
	Thus the constraint is preserved.
	
	\smallskip
	\noindent
	\textbf{Step 2:} \emph{From the constrained slope equation to the scalar
		equation.}
	Since
	$
	v=u_x,
$
	we have
$
	v_t=(u_t)_x,\ 
	v_{xx}=(u_{xx})_x,\ 
	v_{xxxx}=(u_{xxxx})_x.
	$
	Furthermore, we have 
	\begin{align*}
	vv_x
	=
	u_xu_{xx}
	=
	\frac12\big(u_x^2\big)_x.
	\end{align*}
	Substituting these identities into \eqref{eq:constrained-v} gives
	\begin{align*}
	(u_t)_x+(u_{xxxx})_x+(u_{xx})_x
	+\frac12(u_x^2)_x
	=
	\lambda_{u_x}(t)u_x,
	\end{align*}
which implies 
	\begin{align*}
	\left[
	u_t+u_{xxxx}+u_{xx}
	+\frac12(u_x)^2-\lambda_{u_x}(t)u
	\right]_x=0.
	\end{align*}
	Hence there exists a function $c(t)$, independent of $x$, such that
	\begin{equation}\label{eq:constant-time}
		u_t+u_{xxxx}+u_{xx}
		+\frac12(u_x)^2
		=
		\lambda_{u_x}(t)u+c(t).
	\end{equation}
The term $c(t)$ can be removed by changing $u$ by a function of
	time only. Indeed, let
	\begin{align*}
	\widetilde u(t,x)=u(t,x)+h(t),
	\end{align*}
	where $h$ solves
	\begin{align*}
	h'(t)-\lambda_{u_x}(t)h(t)=-c(t).
	\end{align*}
	Since $h$ is independent of $x$, we have 
	$
	\widetilde u_x=u_x,\ 
	\widetilde u_{xx}=u_{xx},\ 
	\widetilde u_{xxxx}=u_{xxxx}.
$
	Consequently, \eqref{eq:constant-time} becomes
	\begin{align*}
	\widetilde u_t+\widetilde u_{xxxx}
	+\widetilde u_{xx}
	+\frac12(\widetilde u_x)^2
	=
	\lambda_{u_x}(t)\widetilde u.
	\end{align*}
	Thus, after fixing the irrelevant spatially constant component of $u$,
	we obtain \eqref{eq:constrained-u}.
	Since $v=u_x$ and $\|v(t)\|_{L^2}=1$, we also have
	$\|u_x(t)\|_{L^2}=1$, for all $t\geq 0$.

	\smallskip
	\noindent
	\textbf{Step 3:} \emph{Converse implication.}
		Conversely, suppose that $u$ satisfies
	\eqref{eq:constrained-u} and
$	\|u_x(t)\|_{L^2}=1$ for all $t\geq 0$. 
Let us set 
$
	v=u_x.
$
	Differentiating \eqref{eq:constrained-u} with respect to
	$x$ gives
	\begin{align*}
	v_t+v_{xxxx}+v_{xx}
	+\frac12\big(v^2\big)_x
	=
	\lambda_v(t)v.
	\end{align*}
	Since
$
	\frac12(v^2)_x=vv_x,
$
	we obtain \eqref{eq:constrained-v}.
	Moreover, we have 
	$
	\|v(t)\|_{L^2}
	=
	\|u_x(t)\|_{L^2}
	=1, 
$ for all $t\geq 0$. 
	Taking the $L^2$-inner product with $v$ as in Step 1 and using
	$\|v(t)\|_{L^2}=1$ yields
	\begin{align*}
	\lambda_v(t)
	=
	\|v_{xx}(t)\|_{L^2}^2
	-\|v_x(t)\|_{L^2}^2.
	\end{align*}
	Since $v=u_x$, we get 
$
	\lambda_{u_x}(t)
	=
	\|u_{xxx}(t)\|_{L^2}^2
	-\|u_{xx}(t)\|_{L^2}^2.
$
	Therefore the constrained slope formulation and the constrained scalar
	formulation are equivalent, up to the addition of a function of time
	to the scalar potential $u$.
\end{proof}

Let us now provide the proof of Theorem \ref{thm-one-d-global}.

	\begin{proof}[Proof of Theorem \ref{thm-one-d-global}]
		We use a Fourier--Galerkin approximation to obtain global solvability results for \eqref{eq:constrained-v}.  The proof is divided into the following steps: 
		
		\smallskip
		
		\noindent
		\textbf{Step 1:} \emph{Fourier--Galerkin approximation.}
			Let
		\begin{align*}
		\widetilde{e}_k(x)=\frac1{\sqrt L}e^{2\pi ikx/L},
		\ \  k\in\mathbb Z,
		\end{align*}
		be the standard orthonormal Fourier basis of $L^2(\mathbb T_L)$.
		For $N\in\mathbb N$, let $P_N$ be the orthogonal projection onto
	$	X_N=\operatorname{span}\{\widetilde{e}_k:|k|\le N\}.$
		Since
$		P_Nv_0\to v_0
		\ \text{ in }\ H^2(\mathbb T_L),$
		we may define
		\begin{align*}
		v_0^N
		=
		\frac{P_Nv_0}{\|P_Nv_0\| _{L^2}}.
		\end{align*}
		For all sufficiently large $N$ this is well-defined, and
		\begin{align*}
		\|v_0^N\| _{L^2}=1,
		\ \ 
		v_0^N\to v_0
		\ \text{ in }\ H^2(\mathbb T_L).
		\end{align*}
		The finitely many remaining $N$ are irrelevant and may be discarded.
		
		We first consider the following finite-dimensional system obtained through the Fourier--Galerkin approximation:
		\begin{equation}\label{eq:galerkin}
			\begin{cases}
				v_t^N
				+
				P_Nv_{xxxx}^N
				+
				P_Nv_{xx}^N
				+
				P_N(v^Nv_x^N)
				=\lambda_{v^N}^N(t)v^N,
				\\
				v^N(0)=v_0^N,
			\end{cases}
		\end{equation}
		where
		\begin{equation}\label{eq:lambdaN}
			\lambda_{v^N}^N(t)
			=
			\|v_{xx}^N(t)\| _{L^2}^2
			-
			\|v_{x}^N(t)\| _{L^2}^2.
		\end{equation}
			Since $X_N$ is finite-dimensional and the nonlinear vector field $$\lambda_{v^N}^N(t)v^N-\left(P_Nv_{xxxx}^N
		+
		P_Nv_{xx}^N
		+
		P_N(v^Nv_x^N)\right)$$  is locally Lipschitz continuous, the Picard--Lindelöf theorem guarantees the existence of a unique maximal solution
		\begin{align*}
		v^N\in C^1([0,T_N);X_N).
		\end{align*}
		
		\smallskip
		
		\noindent
		\textbf{Step 2:} \emph{Preservation of the $L^2$-constraint.}
			Taking the $L^2$-inner product of \eqref{eq:galerkin} with $v^N$ gives
		\begin{align}\label{eqn-001}
		\frac12\frac{d}{dt}\|v^N\| _{L^2}^2
		+
		\|v_{xx}^N\| _{L^2}^2
		-
		\|v_{x}^N\| _{L^2}^2
		+
		(v^Nv_{x}^N,v^N)
		=
		\lambda_{v^N}^N(t)\|v^N\| _{L^2}^2,
		\end{align}
		where we have used the fact that $P_N$ is an orthogonal projection and
		$v^N\in X_N$.
	By periodicity, we obtain  
		\begin{align}\label{eqn-002}
		(v^Nv_{x}^N,v^N)
		=
		\int_{\mathbb T_L}(v^N)^2v_{x}^N dx
		=
		\frac13\int_{\mathbb T_L}\partial_x (v^N)^3 dx
		=0.
		\end{align}
		Using \eqref{eq:lambdaN} and \eqref{eqn-002} in \eqref{eqn-001}, we obtain
		\begin{align*}
		\frac12\frac{d}{dt}\left(\|v^N\| _{L^2}^2-1\right)
		=
		\left(
		\|v_{xx}^N\| _{L^2}^2-\|v_{x}^N\| _{L^2}^2
		\right)
		\left(
		\|v^N\| _{L^2}^2-1
		\right).
		\end{align*}
		Since
	$
		\|v_{0}^N\|_{L^2}^2=1,
	$
		uniqueness of solutions to this scalar ordinary differential equation, together with the equivalence of norms in finite-dimensional spaces, gives
		\begin{equation}\label{eq:constraintN}
			\|v^N(t)\| _{L^2}=1
			\ \text{ for all }\ t\in[0,T_N).
		\end{equation}
		
		\smallskip
		
		\noindent
		\textbf{Step 3:} \emph{Energy estimate.}
			Since $v^N$ is smooth in time and belongs to the finite-dimensional space
		$X_N$, we may take the inner product of \eqref{eq:galerkin} with $v_{t}^N(\cdot)$ to obtain
		\begin{align}\label{eqn-003}
	&	\|v_{t}^N(t)\| _{L^2}^2
		+
		(v_{xxxx}^N(t),v_{t}^N(t))
		+
		(v_{xx}^N(t),v_{t}^N(t))
		+
		(v^N(t)v_{x}^N(t),v_{t}^N(t))
	\nonumber\\&	=
		\lambda_{v^N}^N(t)(v^N(t),v_{t}^N(t)),
		\end{align}
	for a.e. $t\in[0,T]$.	Using \eqref{eq:constraintN}, we infer 
		\begin{align*}
		(v^N(t),v_{t}^N(t))
		=
		\frac12\frac{d}{dt}\|v^N(t)\| _{L^2}^2
		=0,
		\end{align*}
	for a.e. $t\in[0,T]$.	Moreover, periodic integration by parts gives 	for a.e. $t\in[0,T]$
		\begin{align*}
		(v_{xxxx}^N(t),v_{t}^N(t))+(v_{xx}^N(t),v_{t}^N(t))
%		=
%		\frac12\frac{d}{dt}\|v_{xx}^N(t)\| _{L^2}^2
%		-\frac12\frac{d}{dt}\|v_{x}^N(t)\| _{L^2}^2
		=
		\frac12\frac{d}{dt}
		\left(
		\|v_{xx}^N(t)\| _{L^2}^2-\|v_{x}^N(t)\| _{L^2}^2
		\right).
		\end{align*}
	Therefore, we infer from \eqref{eqn-003} that 
		\begin{equation}\label{eq:energyN}
			\frac12\frac{d}{dt}
			\left(
			\|v_{xx}^N(t)\| _{L^2}^2-\|v_{x}^N(t)\| _{L^2}^2
			\right)
			+
			\|v_{t}^N(t)\| _{L^2}^2
			=
			-(v^N(t)v_{x}^N(t),v_{t}^N(t)),
		\end{equation}
for a.e. $t\in[0,T]$.		Let us define 
		\begin{align}\label{eqn-eN}
		E_N(t)
		=
		\frac12
		\left(
		\|v_{xx}^N(t)\| _{L^2}^2-\|v_{x}^N(t)\| _{L^2}^2
		\right), \ t\in[0,T]. 
		\end{align}
	Therefore, from \eqref{eq:energyN}, we immediately have 
		\begin{equation}\label{eq:energy-ineq-N}
		E_N'(t)+\|v_{t}^N(t)\| _{L^2}^2
		=
		-(v^N(t)v_{x}^N(t),v_{t}^N(t))\leq 
		\frac12\|v_{t}^N(t)\| _{L^2}^2
		+
		\frac12\|v^N(t)v_{x}^N(t)\| _{L^2}^2,
		\end{equation}
	for a.e. $t\in[0,T]$, where we have used Young's inequality.

		\smallskip
		
		\noindent
		\textbf{Step 4:} \emph{Estimate of the nonlinear term.}
		In one dimension, Agmon's  inequality gives
		\begin{align}\label{eqn-Agmon}
		\|v\|_\infty^2
		\le
		C_L
		\left(
		\|v\| _{L^2}^2+\|v\| _{L^2}\|v_x\| _{L^2}
		\right),
		\ \
		v\in H^1(\mathbb T_L).
		\end{align}
		Applying this to $v^N$ and using \eqref{eq:constraintN}, we find 
		\begin{align*}
		\|v^N\|_{L^{\infty}}^2
		\le
		C\left(1+\|v_{x}^N\| _{L^2}\right),
		\end{align*}
	so that 
		\begin{align}\label{eqn-007}
			\|v^Nv_{x}^N\| _{L^2}^2
			\le
			\|v^N\|_{L^{\infty}}^2\|v_{x}^N\| _{L^2}^2\le
			C
			\left(
			\|v_{x}^N\| _{L^2}^2
			+
			\|v_{x}^N\| _{L^2}^3
			\right).
		\end{align}
		Consequently, we infer from \eqref{eq:energy-ineq-N} that 
		\begin{equation}\label{eq:EprimeN}
			E_N'(t) +\frac12\|v_{t}^N(t)\| _{L^2}^2
			\le
			C
			\left(
			\|v_{x}^N(t)\| _{L^2}^2
			+
			\|v_{x}^N(t)\| _{L^2}^3
			\right),
		\end{equation}
		for a.e. $t\in[0,T]$. 
		
		\smallskip
		
		\noindent
		\textbf{Step 5:} \emph{A priori $H^2$ estimate.}
			Since $\|v^N\| _{L^2}=1$, periodic integration by parts gives
		\begin{align}\label{eqn-004}
		\|v_{x}^N\| _{L^2}^2
		=
		-(v^N,v_{xx}^N)
		\le
		\|v_{xx}^N\| _{L^2}.
		\end{align}
		Let us set
		\begin{align*}
		Z_N (t):=\|v_{xx}^N(t) \| _{L^2}, \ t\in[0,T]. 
		\end{align*}
		Then, from \eqref{eqn-eN}, we have 
		\begin{align}\label{eqn-eN-bound}
		E_N(t)
		\ge
		\frac12(Z_N^2(t)-Z_N(t))
		=
		\frac12\left(Z_N(t)-\frac12\right)^2-\frac18, \ t\in[0,T]. 
		\end{align}
		Let us define
		\begin{align*}
		F_N(t)=E_N(t)+\frac18, \ t\in[0,T].
		\end{align*}
			Then, it is immediate from \eqref{eqn-eN-bound} that 
		\begin{align}\label{eqn-006}
		F_N(t)\ge0
	\ \text{	and }\ 
		F_N(t)
		\ge
		\frac12
		\left(Z_N(t)-\frac12\right)^2.
		\end{align}
	Therefore, we deduce for $t\in[0,T]$ that 
		\begin{align*}
		4F_N(t)+\frac12
		\geq
		2\left(Z_N(t)-\frac12\right)^2+\frac12
		=2Z_N^2(t)-2Z_N(t)+1= Z_N^2(t)+\left(Z_N(t)-1\right)^2\geq  Z_N^2(t).
		\end{align*}
		Consequently, we find 
		\begin{align*}
		Z_N^2(t)
		\le
		4F_N(t)+\frac12\Longrightarrow 
		Z_N(t)
		\le
		C(1+F_N(t))^{1/2}, \ t\in[0,T]. 
		\end{align*}
		Also, from \eqref{eqn-004}, we further have 
		\begin{align*}
		\|v_{x}^N(t)\| _{L^2}^2\le Z_N(t)
		\le
		C(1+F_N(t))^{1/2},
		\end{align*}
		and hence
		\begin{align*}
		\|v_{x}^N(t)\| _{L^2}^3
		\le
		C(1+F_N(t))^{3/4}, \ t\in[0,T]. 
		\end{align*}
		Therefore, we deduce 
		\begin{align*}
		\|v_{x}^N(t)\| _{L^2}^2+\|v_{x}^N(t)\| _{L^2}^3
		\le
		C(1+F_N(t))^{3/4}, \ t\in[0,T]. 
		\end{align*}
		Since $F_N'=E_N'$, \eqref{eq:EprimeN} yields
		\begin{align*}
		F_N'(t)
		\le
		C(1+F_N(t))^{3/4}, \ \text{ for a.e. } \  t\in[0,T].  
		\end{align*}
		Putting
	$
		G_N(t)=1+F_N(t),
	$
		we obtain
		\begin{align*}
		G_N'(t)\le CG_N(t)^{3/4}, \ \text{ for a.e. } \  t\in[0,T] 
\end{align*}
which implies 
\begin{align*}
		G_N(t)^{1/4}
		\le
		G_N(0)^{1/4}+CT, \ \text{ for all }\ t\in[0,T]. 
		\end{align*}
		Since
	$
		v_{0}^N\to v_0
		\ \text{ in }\ H^2(\mathbb T_L),
	$
		the sequence $G_N(0)$ is uniformly bounded. Therefore, for every
		$T>0$,
		\begin{align*}
		\sup_{0\le t\le T}F_N(t)\le C_T,
		\end{align*}
		where $C_T$ is independent of $N$.
		We recall  from \eqref{eqn-006} that
		\begin{align*}
 	\frac12
		\left(
		Z_N(t)-\frac12
		\right)^2 \leq F_N(t)
		\le C_T \ \Longrightarrow 
		\left|
		Z_N(t)-\frac12
		\right|
		\le
		\sqrt{2C_T}.
		\end{align*}
		Consequently, we get 
		\begin{align}\label{eqn-est-10}
		\sup_{0\le t\le T}	\|v_{xx}^N(t)\| _{L^2}
		=
		\sup_{0\le t\le T}	Z_N(t)
		\le
		\frac12+\sqrt{2C_T}
		\le C_T.
		\end{align}
		Moreover, using the fact that  $\|v^N(t)\| _{L^2}=1$ for all $t\in[0,T]$ and \eqref{eqn-004}, we obtain 
		\begin{equation}\label{eq:H2N}
				\sup_{0\le t\le T}
				\|v^N(t)\|_{H^2(\mathbb T_L)}
				\le \sup_{0\le t\le T}
				\left(
				\|v^N(t)\| _{L^2}^2
				+
				\|v_{x}^N(t)\| _{L^2}^2
				+
				\|v_{xx}^N(t)\| _{L^2}^2
				\right) \leq C_T.
		\end{equation}

		\smallskip
		
		\noindent
		\textbf{Step 6:} \emph{Estimate for $v_{t}^N$.}	Integrating \eqref{eq:EprimeN} over $(0,T)$ gives
		\begin{align*}
		E_N(T)-E_N(0)
		+
		\frac12\int_0^T\|v_{t}^N(t)\| _{L^2}^2 dt
		\le
		C\int_0^T
		\left(
		\|v_{x}^N(t)\| _{L^2}^2
		+
		\|v_{x}^N(t)\| _{L^2}^3
		\right) dt.
		\end{align*}
		Since
	$
		E_N(T)\ge-\frac18
		$
		and \eqref{eq:H2N} holds, we obtain
		\begin{equation}\label{eq:utN}
				\int_0^T\|v_{t}^N(t)\| _{L^2}^2 dt
				\le C_T,
		\end{equation}
		where $C_T$ is independent of $N$.
		
		\smallskip
		
		\noindent
		\textbf{Step 7:} \emph{Uniform $H^4$ estimate and the multiplier.}
			By the definition of $\lambda_{v^N}^N$ and using  \eqref{eq:H2N}, we get 
		\begin{align*}
		|\lambda_{v^N}^N(t)|
		=
	\left|	\|v_{xx}^N(t)\| _{L^2}^2-	\|v_{x}^N(t)\| _{L^2}^2\right| \leq 
		\|v_{x}^N(t)\| _{L^2}^2
		+
		\|v_{xx}^N(t)\| _{L^2}^2\leq C_T,
		\end{align*}
so that 
		$	\|\lambda_{v^N}^N\|_{L^\infty(0,T)}
			\le C_T.$ 
		Hence, from  \eqref{eq:galerkin}, \eqref{eqn-007}, \eqref{eqn-est-10} and \eqref{eq:H2N}, we find 
		\begin{align*}
		P_N(v_{xxxx}^N)
		=
		-v_{t}^N-P_N(v_{xx}^N)-P_N(v^Nv_{x}^N)-\lambda_{v^N}^Nv^N,
		\end{align*}
	and 
		\begin{align*}
		\|P_N(v_{xxxx}^N)\|_{L^2(0,T;L^2)}
		\le C_T.
		\end{align*}
		Since $v^N\in X_N$, $P_N$ commutes with $\partial_x$ and therefore
	$
		P_N(v_{xxxx}^N)
		=
		v_{xxxx}^N,
	$
	so that 
		\begin{align}\label{eqn-est-11}
		\|v_{xxxx}^N\|_{L^2(0,T;L^2)}
		\le C_T.
		\end{align}
		We now show that this estimate, together with the previously established \(H^2\)-estimate \eqref{eq:H2N}, implies a uniform \(H^4\)-bound. For each \(t\in[0,T]\), since \(v^N(t)\) is periodic, integration by parts and Young's inequality yield 
		\begin{align*}
		\|v_{xxx}^N(t)\|_{L^2}^2
		&=
		-\int_{\mathbb T_L}v_{xx}^N(t)v_{xxxx}^N(t)\,dx
		\le
		\|v_{xx}^N(t)\|_{L^2}\|v_{xxxx}^N(t)\|_{L^2}
		\nonumber\\&	\le
		\frac12\|v_{xx}^N(t)\|_{L^2}^2
		+
		\frac12\|v_{xxxx}^N(t)\|_{L^2}^2.
	\end{align*}
	Therefore, using the standard definition of the periodic \(H^4\)-norm, we find 
		\begin{align*}
			\int_0^T	\|v^N(t)\|_{H^4}^2dt 
			&\le
			C_L	\int_0^T\left(	\|v^N(t)\|_{H^2}^2+	\|v_{xxx}^N(t)\|_{L^2}^2+	\|v_{xxxx}^N(t)\|_{L^2}^2\right)dt\\
			&\le
			C_L	\int_0^T\left(
			\|v^N(t)\|_{H^2}^2
			+
			\|v_{xxxx}^N(t)\|_{L^2}^2
			\right)dt\leq C_T,
		\end{align*}
where we have used \eqref{eq:H2N} and \eqref{eqn-est-11}.  	We finally have 
\begin{equation}\label{eq:H4N}
	\|v^N\|_{L^2(0,T;H^4(\mathbb T_L))}
	\le C_T,
\end{equation}
where $C_T$ is independent of $N$.

		\smallskip
		
		\noindent
		\textbf{Step 8:} \emph{Global existence of the Fourier--Galerkin solutions.}
	The estimate \eqref{eq:H2N} is uniform with respect to $N$. Suppose,
	for contradiction, that $T_N<\infty$ for some $N$. By the
	finite-dimensional Picard--Lindel\"of theorem, the maximal solution
	$v^N$ can be continued beyond $T_N$ provided that
	\begin{align*}
	\limsup_{t\uparrow T_N}\|v^N(t)\|_{X_N}<\infty.
	\end{align*}
	Since $X_N$ is finite-dimensional, all norms on $X_N$ are equivalent.
	In particular, there exists a constant $C_N>0$ such that
$
	\|z\|_{X_N}\leq C_N\|z\|_{H^2},
	\ z\in X_N.
$
	Consequently, the uniform estimate \eqref{eq:H2N} yields
	\begin{align*}
	\sup_{0\leq t<T_N}\|v^N(t)\|_{X_N}
	\leq
	C_N\sup_{0\leq t<T_N}\|v^N(t)\|_{H^2}
	<\infty.
	\end{align*}
	Thus, the finite-dimensional norm of $v^N(t)$ remains bounded as
	$t\uparrow T_N$. The continuation criterion for the finite-dimensional
	ODE then implies that $v^N$ can be extended to an interval
	$[0,T_N+\varepsilon]$ for some $\varepsilon>0$, contradicting the
	maximality of $T_N$.	Therefore,
	$T_N=\infty$
	for every $N\in\mathbb N$. Hence each Galerkin approximation admits a
	global solution.
		
		\smallskip
		
		\noindent
		\textbf{Step 9:} \emph{Compactness.}
			Fix $T>0$. From \eqref{eq:H2N}, \eqref{eq:utN} and \eqref{eq:H4N} , we infer 
		\begin{align*}
	&	\{v^N\}
		\ \text{ is bounded in }
	L^\infty(0,T;H^2(\mathbb T_L))\cap 	L^2(0,T;H^4(\mathbb T_L))
	\ \text{
		and }\\
	&	\{v_{t}^N\}
		\ \text{ is bounded in }
		L^2(0,T;L^2(\mathbb T_L)).
	\end{align*}
By the Banach--Alaoglu theorem, there exists a subsequence,
still denoted by $\{v^N\}$, and some
$v\in L^\infty(0,T;H^2(\mathbb T_L))\cap L^2(0,T;H^4(\mathbb T_L)),$ such that
\begin{equation}\label{eqn-strong-conv}
\begin{aligned}
v^N&\xrightarrow{w^*} v
\ \text{ in }\ L^\infty(0,T;H^2(\mathbb T_L)),\\
v^N&\xrightarrow{w} v
\ \text{ in }\ L^2(0,T;H^4(\mathbb T_L)),\\
v_{t}^N&\xrightarrow{w} v_t
\ \text{ in }\ L^2(0,T;L^2(\mathbb T_L)).
\end{aligned}
\end{equation}
		Since
	$
		H^4(\mathbb T_L)
		\hookrightarrow\hookrightarrow
		H^2(\mathbb T_L)
		\hookrightarrow
		L^2(\mathbb T_L),
	$ and $
	H^2(\mathbb T_L)
	\hookrightarrow\hookrightarrow
	H^1(\mathbb T_L)
	\hookrightarrow
	L^2(\mathbb T_L),
	$
		the Aubin--Lions--Simons compactness lemma (\cite[Corollary 5]{JS-87}) gives, after passing to a subsequence,
		\begin{align}
		v^N&\to v
		\ \text{ strongly in }\ 
		L^2(0,T;H^2(\mathbb T_L)),\\
		v^N&\to v
		\ \text{ strongly in }\ 
		C([0,T];H^1(\mathbb T_L)),
		\end{align}
	Using the Sobolev embedding $H^1(\mathbb T_L)
	\hookrightarrow
	L^{\infty}(\mathbb T_L) $, we obtain 
		\begin{align}
			\sup_{t\in[0,T]}\|v^N(t)- v(t)\|_{L^{\infty}}\to 0\ \text{ as }\ N\to\infty  \label{strong-2}.
		\end{align}
	Since 
		\begin{align}
			\|v^Nv_{x}^N-vv_x\| _{L^2}
			\le
			\|v^N-v\|_{L^{\infty}}\|v_{x}^N\| _{L^2}
			+
			\|v\|_{L^{\infty}}\|v_{x}^N-v_x\| _{L^2},
		\end{align}
		and the right-hand side converges to zero in $L^{\infty}(0,T)$. Therefore, we have 
		\begin{align*}
			v^Nv_{x}^N\to vv_x \ \text{ strongly in }\ L^{\infty}(0,T;L^2(\mathbb T_L)). 
		\end{align*}
		
		\smallskip
		
		\noindent
		\textbf{Step 10:} \emph{Identification of the Lagrange multiplier.}
	We infer  from \eqref{eq:H2N} that 
		\begin{align*}
		\|\lambda_{v^N}^N\|_{L^\infty(0,T)}\le C_T.
		\end{align*}
		Hence, after passing to a subsequence by using the Bolzano--Weierstrass Theorem, we get 
		\begin{align*}
		\lambda_{v^N}^N\xrightarrow{w^*}\lambda
		\ \text{ in }\ L^\infty(0,T).
		\end{align*}
		However, we need to identify $\lambda$. 	The strong convergence
	given in \eqref{strong-2} 
		implies
		\begin{align*}
		v_{x}^N\to v_x
		\quad\text{and}\quad
		v_{xx}^N\to v_{xx}
		\qquad\text{in }L^2(0,T;L^2).
	\end{align*}
		Therefore
		\begin{align*}
		\|v_{x}^N\| _{L^2}^2
		\to
		\|v_x\| _{L^2}^2 \ \text{ and }\ 
		\|v_{xx}^N\| _{L^2}^2
		\to
		\|v_{xx}\| _{L^2}^2
		\ \text{ in }\ L^1(0,T),
		\end{align*}
	so that 
		\begin{align*}
		\lambda_{v^N}^N
		=
	\|v_{xx}^N\| _{L^2}^2-	\|v_{x}^N\| _{L^2}^2
		\to
	\|v_{xx}\| _{L^2}^2-	\|v_x\| _{L^2}^2
		\end{align*}
		strongly in $L^1(0,T)$.
	Since $\lambda_{v^N}^N\xrightarrow{w^*}\lambda
	\ \text{ in }\ L^\infty(0,T)$, the limit is uniquely identified as
		\begin{align*}
			\lambda(t)=
			\|v_{xx}(t)\| _{L^2}^2-\|v_x(t)\| _{L^2}^2=:\lambda_v(t),
		\end{align*}
		for a.e. $t\in(0,T)$.
		
		\smallskip
		
		\noindent
		\textbf{Step 11:} \emph{Passage to the limit in the equation.}
		Let $\psi\in C^\infty(\mathbb T_L\times[0,T])$ with $\psi(T)=0$ be periodic in $x$.
		The Galerkin equation, tested against $P_N\psi$, gives
		\begin{align}
			&\int_0^T
			\left[
			-(v^N,P_N\psi_t) +	(v_{xx}^N,(P_N\psi)_{xx})	-
			(v_{x}^N,(P_N\psi)_x)	+ (v^Nv_{x}^N,P_N\psi)	-	\lambda_{v^N}^N(t)(v^N,P_N\psi)	\right] dt	\\	&\quad	=	(v_{0}^N,P_N\psi(0)).
		\end{align}
		Using the convergences obtained above and the fact that
	$
		P_N\psi\to\psi
		\ \text{ in }\ C^\infty(\mathbb T_L\times[0,T]),
	$
		we may pass to the limit and obtain
		\begin{align*}
		\begin{aligned}
			&\int_0^T
			\Big[
			-(v,\psi_t)
			+
			(v_{xx},\psi_{xx})
			-
			(v_x,\psi_x)
			+
			(vv_x,\psi)
			-
			\lambda_v(t)(v,\psi)
			\Big] dt
			=
			(v_0,\psi(0)).
		\end{aligned}
		\end{align*}
		Therefore, we conclude 
	\begin{align}	\label{eqn-strong}
		v_t+v_{xxxx}+v_{xx}+vv_x=\lambda_v(t)v
		\end{align}
		in the distributional sense.
		Moreover, we have the following regularity by an application of the Lions--Magenes lemma (\cite[Theorem 3.1, Chapter 1]{JLL+EM-I-72})
			\begin{align}\label{eqn-reg}
		v\in
		C([0,T];H^2(\mathbb T_L))
		\cap
		H^1(0,T;L^2(\mathbb T_L))
		\cap
		L^2(0,T;H^4(\mathbb T_L)),
		\end{align}
		which implies that $v$ is a global strong solution and \eqref{eqn-strong} is satisfied in $L^2(0,T;L^2(\mathbb T_L))$.

		\smallskip
		
		\noindent
		\textbf{Step 12:} \emph{Preservation of the constraint.}
	From the convergence \eqref{strong-2}, we infer 
		$
		v\in C([0,T];H^2(\mathbb T_L)).$
		In particular,
	$
		v\in C([0,T];L^2(\mathbb T_L)).
	$
		Since
	$
		\|v^N(t)\| _{L^2}=1
	$
		for every  $t\in[0,T]$, passage to the limit gives
	$
		\|v(t)\| _{L^2}=1
		\ \text{ for every }\ t\in[0,T].
		$
		Since $T>0$ is arbitrary, we deduce 
		\begin{equation}\label{eq:constraint-final}
				\|v(t)\| _{L^2}=1
			\	 \text{ for every }\ t\ge0.
		\end{equation}

		\smallskip
		
		\noindent
		\textbf{Step 13:} \emph{Uniqueness.}
	Let $v_1$ and $v_2$ be two solutions of \eqref{eq:constrained-v} with the stated regularity in \eqref{eqn-reg} and the
		same initial datum. Let us take 
	$
		w=v_1-v_2.
	$
		Then $w$ satisfies 
		\begin{align*}
		w_t+w_{xxxx}+w_{xx}
		+
		v_1{v_1}_x-v_2{v_2}_x
		=
		\lambda_{v_1}v_1-\lambda_{v_2}v_2
		\end{align*}
in $L^2(0,T;L^2(\mathbb T_L))$, 		where
		\begin{align*}
		\lambda_{v_1}
		=
	\|{v_1}_{xx}\| _{L^2}^2-	\|{v_1}_x\| _{L^2}^2\ \text{ and }\ 
		\lambda_{v_2}
		=
	\|{v_2}_{xx}\| _{L^2}^2-	\|{v_2}_x\| _{L^2}^2. 
		\end{align*}
		We write
	\begin{align*}
		v_1{v_1}_x-v_2{v_2}_x
		=
		v_1 w_x+{v_2}_xw
	\ \text{ and }\ 
		\lambda_{v_1} v_1-\lambda_{v_2}v_2
		=
		\lambda_{v_1} w+(\lambda_{v_1}-\lambda_{v_2})v_2.
	\end{align*}
		Taking the $L^2$-inner product with $w$ gives
		\begin{align}\label{eqn-010}
		&	\frac12\frac{d}{dt}\|w\| _{L^2}^2
			+
			\|w_{xx}\| _{L^2}^2-\|w_x\| _{L^2}^2
			+
			(v_1 w_x,w)
			+
			({v_2}_xw,w)
	\nonumber\\&	=
			\lambda_{v_1}\|w\| _{L^2}^2
			+
			(\lambda_{v_1}-\lambda_{v_2})(v_2,w).
		\end{align}
		On every finite interval $[0,T]$, the $H^2$ bounds imply
		\begin{align}\label{eqn-008}
		\|v_1\|_{W^{1,\infty}}
		+
		\|v_2\|_{W^{1,\infty}}
		+
		|\lambda_{v_1}|+|\lambda_{v_2}|
		\le C_T.
		\end{align}
		Using the Cauchy--Schwarz and Young  inequalities, we find 
		\begin{align}\label{eqn-009}
		\|w_x\| _{L^2}^2=-\int_{\mathbb T_L}ww_{xx}dx\leq \|w\|_{L^2}\|w_{xx}\|_{L^2}
		\le
		\frac{1}{4}\|w_{xx}\| _{L^2}^2+\|w\| _{L^2}^2.
		\end{align}
	Using H\"older's, Gagliardo--Nirenberg's
	 and Young's inequalities, \eqref{eqn-009} and \eqref{eqn-008}, we deduce 
		\begin{align}
			|(v_1 w_x,w)|
		&	\le
			\|v_1\|_{L^{\infty}}\|w_x\| _{L^2}\|w\| _{L^2}\leq C_T\|w_{xx}\|_{L^2}^{1/2}\|w\|_{L^2}^{3/2}\le
			\frac14\|w_{xx}\| _{L^2}^2
			+
			C_T\|w\| _{L^2}^2,\\
				|({v_2}_xw,w)|&\leq \|{v_2}_{x}\|_{L^{\infty}}\|w\|_{L^2}^2
			\le
			C_T\|w\| _{L^2}^2.
		\end{align}
	Furthermore, we find 
		\begin{align*}
			|\lambda_{v_1}-\lambda_{v_2}|
			&\le 	(\|{v_1}_{xx}\| _{L^2}+\|{v_2}_{xx}\| _{L^2})\|w_{xx}\| _{L^2}+
			(\|{v_1}_x\| _{L^2}+\|{v_2}_x\| _{L^2})\|w_x\| _{L^2}
			\\
			&\le
			C_T\left(\|w_{xx}\| _{L^2}+\|w_{x}\| _{L^2}\right),
		\end{align*}
so that 
		\begin{align}
			|(\lambda_{v_1}-\lambda_{v_2})(v_2,w)|
			&\le
			C_T
			\left(
			\|w_x\| _{L^2}+\|w_{xx}\| _{L^2}
			\right)\|w\| _{L^2}
		\leq 
			\frac14\|w_{xx}\| _{L^2}^2
			+
			C_T\|w\| _{L^2}^2.
		\end{align}
		Finally, we have 
		\begin{align*}
		|\lambda_{v_1}|\|w\| _{L^2}^2
	\leq \left(	\|{v_1}_{xx}\| _{L^2}^2+	\|{v_1}_x\| _{L^2}^2\right)	\|w\| _{L^2}^2 \le C_T\|w\| _{L^2}^2.
		\end{align*}
		Combining the above estimates and using it in \eqref{eqn-010}, 
		we obtain
		\begin{align*}
		\frac{d}{dt}\|w(t)\| _{L^2}^2+\frac12\|w_{xx}\|_{L^2}^2
		\le
		C_T\|w(t)\| _{L^2}^2.
		\end{align*}
		Since
	$
		w(0)=0,
	$
		Gr\"onwall's inequality gives
	$
		w(t)=0
		\ \text{ for all }\ t\in[0,T].
		$
		Since $T>0$ is arbitrary,
	$
		v_1=v_2
		\ \text{ on }\ [0,\infty),
	$
		which completes the proof.
	\end{proof}

	\section{Proof of Theorem \ref{thm-local}}	\label{sec-4}
	
	In this section, we establish the global well-posedness of the two-dimensional KS equation by proving Theorem \ref{thm-local}.
	\begin{proposition}[Equivalence of the scalar and vector constrained formulations]\label{prop-equivalence-2}
		Let
	$
		\mathbb T_L^d=(\mathbb R/L\mathbb Z)^d,
		\  d\in\{2,3\}.
	$
		Assume that all functions are sufficiently smooth and periodic in the
		spatial variables. Let $\psi_0$ be a periodic scalar function and
		set
		\begin{align*}
		\boldsymbol{v}_0=\nabla\psi_0,
		\ 
		\|\nabla\psi_0\|_{L^2(\mathbb T_L^d;\mathbb{R}^d)}=1.
		\end{align*}
		Consider the scalar equation
		\begin{equation}\label{eq:scalar}
			\psi_t+\Delta^2\psi+\Delta\psi
			+\frac12|\nabla\psi|^2
			=
			\lambda_{\boldsymbol{v}}(t)\psi
		\end{equation}
		with
		\begin{align}\label{eqn-lambda}
		\lambda_{\boldsymbol{v}}(t)
		=
		\|\Delta \boldsymbol{v}(t)\|_{L^2}^2
		-\|\nabla \boldsymbol{v}(t)\|_{L^2}^2
		-\frac12
		\int_{\mathbb T_L^d}
		|\boldsymbol{v}(t)|^2\operatorname{div}\boldsymbol{v}(t) dx,
		\  t\in[0,T], \ 
		\boldsymbol{v}=\nabla\psi.
		\end{align}
		Then $\boldsymbol{v}=\nabla\psi$ satisfies
		\begin{equation}\label{eq:vector}
			\left\{
			\begin{aligned}
				\boldsymbol{v}_t+\Delta^2\boldsymbol{v}+\Delta \boldsymbol{v}+(\boldsymbol{v}\cdot\nabla)\boldsymbol{v}
				&=\lambda_{\boldsymbol{v}}(t)\boldsymbol{v},\\
				\boldsymbol{v}(0)&=\boldsymbol{v}_0,
			\end{aligned}
			\right.
		\end{equation}
		and
	\begin{equation}\label{eqn-invariant}
		\|\boldsymbol{v}(t)\|_{L^2}=1
		\ \text{ for all }\ t\geq 0
\end{equation}
		for which the solution exists.
		
		Conversely, suppose that $\boldsymbol{v}$ is a sufficiently regular periodic solution
		of \eqref{eq:vector} satisfying
		\begin{align*}
		\|\boldsymbol{v}_0\|_{L^2}=1,
		\ \
		\operatorname{curl}\boldsymbol{v}_0=\boldsymbol{0},
		\ \ 
		\int_{\mathbb T_L^d}\boldsymbol{v}_0(x) dx=\boldsymbol{0},
		\end{align*}
		and that the gradient structure is preserved, that is,
		\begin{align}\label{eqn-curl}
		\operatorname{curl}\boldsymbol{v}(t)=\boldsymbol{0},
		\ \ 
		\int_{\mathbb T_L^d}\boldsymbol{v}(t,x) dx=\boldsymbol{0}.
		\end{align}
		Then there exists a periodic scalar function $\psi$ such that
	$
		\boldsymbol{v}=\nabla\psi,
	$
		and, after adding a suitable function of time to $\psi$, the function
		$\psi$ satisfies \eqref{eq:scalar} and
		\begin{align*}
		\|\nabla\psi(t)\|_{L^2}=1 	\ \text{ for all }\ t\geq 0. 
		\end{align*}
		Thus the two formulations are equivalent on the class of periodic
		gradient vector fields.
	\end{proposition}

	\begin{proof}
		We divide the proof into several steps.
		
		\smallskip
		\noindent
		\textbf{Step 1:} \emph{From the scalar equation to the vector equation.}
			Let
	$\boldsymbol{v}=\nabla\psi.$
		Taking the spatial gradient of \eqref{eq:scalar}, we obtain
		\begin{align*}
		\nabla\psi_t+\nabla\Delta^2\psi
		+\nabla\Delta\psi
		+\frac12\nabla|\nabla\psi|^2
		=
		\lambda_{\boldsymbol{v}}(t)\nabla\psi.
		\end{align*}
		Since
	$
		\nabla\psi_t=\boldsymbol{v}_t,\
		\nabla\Delta^2\psi=\Delta^2\boldsymbol{v},
		\ 
		\nabla\Delta\psi=\Delta \boldsymbol{v},
	$
		it remains to consider the nonlinear term.
			For $k=1,\ldots,d$, we have 
		\begin{align*}
		\begin{aligned}
			\left[(\boldsymbol{v}\cdot\nabla)\boldsymbol{v}\right]_k
			&=
			\sum_{\ell=1}^d v_\ell\frac{\partial v_k}{\partial x_\ell} =
			\sum_{\ell=1}^d
			\frac{\partial\psi}{\partial x_\ell}
			\frac{\partial^2\psi}{\partial x_{\ell}\partial x_k}=
			\frac12\frac{\partial}{\partial x_k}
			\left(
			\sum_{\ell=1}^d
			\left(\frac{\partial \psi}{\partial_\ell}\right)^2
			\right)=
			\frac12\frac{\partial}{\partial x_k}|\nabla\psi|^2.
		\end{aligned}
		\end{align*}
		Therefore,
	$
		(\boldsymbol{v}\cdot\nabla)\boldsymbol{v}
		=
		\frac12\nabla|\nabla\psi|^2.
	$
		Consequently, we obtain  \eqref{eq:vector} and  since
		$
		\boldsymbol{v}(0)=\nabla\psi_0=\boldsymbol{v}_0,
	$
		the initial condition is also preserved.
		
		\smallskip
		\noindent
		\textbf{Step 2:} \emph{The initial constraint implies the constraint for all
			times.} 	We assume that 
		$
		\|\boldsymbol{v}_0\|_{L^2}=1.
	$
		We now prove that
	$
		\|\boldsymbol{v}(t)\|_{L^2}=1\ \text{ for all }\ t\geq 0
	$
		is preserved by the evolution.
		Take the $L^2$-inner product of the vector equation with
		$\boldsymbol{v}$. Using periodicity and integration by parts, we get 
		\begin{align*}
		(\boldsymbol{v},\Delta^2\boldsymbol{v})_{L^2}
		=
		\|\Delta \boldsymbol{v}\|_{L^2}^2
	 \ \text{ and }\ 
		(\boldsymbol{v},\Delta \boldsymbol{v})_{L^2}
		=
		-\|\nabla \boldsymbol{v}\|_{L^2}^2.
		\end{align*}
		Furthermore, we have 
		\begin{align*}
			\int_{\mathbb T_L^d}
			(\boldsymbol{v}\cdot\nabla)\boldsymbol{v}\cdot \boldsymbol{v} dx
			=
			\frac12
			\int_{\mathbb T_L^d}
			\boldsymbol{v}\cdot\nabla |\boldsymbol{v}|^2 dx=
			-\frac12
			\int_{\mathbb T_L^d}
			|\boldsymbol{v}|^2\operatorname{div}\boldsymbol{v} dx.
		\end{align*}
		Hence, we obtain 
		\begin{align*}
		\frac12\frac{d}{dt}\|\boldsymbol{v}(t)\|_{L^2}^2
		+
		\|\Delta \boldsymbol{v}(t)\|_{L^2}^2
		-
		\|\nabla \boldsymbol{v}(t)\|_{L^2}^2
		-
		\frac12
		\int_{\mathbb T_L^d}
		|\boldsymbol{v}(t)|^2\operatorname{div}\boldsymbol{v}(t) dx
		=
		\lambda_{\boldsymbol{v}}(t)\|\boldsymbol{v}(t)\|_{L^2}^2,
		\end{align*}
	for all $t\in[0,T]$.	By the definition of $\lambda_{\boldsymbol{v}}(t)$ given in \eqref{eqn-lambda}, we deduce  for all $t\in[0,T]$ that 
		\begin{align*}
		\frac12\frac{d}{dt}\left(\|\boldsymbol{v}(t)\|_{L^2}^2-1\right)
		=
		\lambda_{\boldsymbol{v}}(t)
		\left(
		\|\boldsymbol{v}(t)\|_{L^2}^2-1
		\right).
		\end{align*}
		The unique solution of this scalar linear ODE with initial condition
		$\|\boldsymbol{v}(0)\|_{L^2}^2-1=0$ is
		$
		\|\boldsymbol{v}(t)\|_{L^2}^2-1=0
	$ for all $t\geq 0$. 
		Consequently, we deduce \eqref{eqn-invariant}.
		Thus the unit $L^2$-sphere is invariant under the vector evolution.
		
		\smallskip
		\noindent
		\textbf{Step 3:} \emph{The gradient structure and the zero-mean condition.}
		Since
	$
		\boldsymbol{v}=\nabla\psi,
	$
		we automatically have
	$
		\operatorname{curl}\boldsymbol{v}=0.
	$
		Moreover, periodicity of $\psi$ implies
		\begin{align*}
		\int_{\mathbb T_L^d}\boldsymbol{v}(t,x) dx
		=
		\int_{\mathbb T_L^d}\nabla\psi(t,x) dx
		=\boldsymbol{0}.
		\end{align*}
		Thus the conditions given in \eqref{eqn-curl} are satisfied. 
		These conditions are important in the converse direction because, on
		the torus, they characterize the periodic gradient fields.
		
		\smallskip
		\noindent
		\textbf{Step 4:} \emph{From the vector equation back to a scalar potential.}
		Now suppose that $\boldsymbol{v}$ is a sufficiently regular periodic solution of
		\eqref{eq:vector} satisfying	\eqref{eqn-curl}.
		For $d=2$, the first condition in \eqref{eqn-curl} means
	$
		\partial_1v_2-\partial_2v_1=0.
	$
		For $d=3$, it means
	$
		\nabla\times \boldsymbol{v}=0.
	$
		We claim that there exists a periodic scalar function $\psi$ such
		that
	$
		\boldsymbol{v}=\nabla\psi.
	$
		Indeed, write the Fourier expansion
		\begin{align*}
	\boldsymbol{v}(t,x)
		=
		\sum_{k\in\mathbb Z^d}
		\widehat{\boldsymbol{v}}_k(t)
		e^{2\pi i k\cdot x/L}.
		\end{align*}
		The zero-mean condition gives
	$
		\widehat{\boldsymbol{v}}_0(t)=0.
	$
		The curl-free condition gives, for every $k\neq0$,
	$
		k\times\widehat{\boldsymbol{v}}_k(t)=\boldsymbol{0}.
	$
		Hence $\widehat{\boldsymbol{v}}_k(t)$ is parallel to $k$. Thus there exists a scalar
		$\widehat\psi_k(t)$ such that
		\begin{align*}
		\widehat{\boldsymbol{v}}_k(t)
		=
		\frac{2\pi i}{L}
		k\,\widehat\psi_k(t).
		\end{align*}
	By	defining
		\begin{align*}
		\psi(t,x)
		=
		\sum_{k\neq0}
		\widehat\psi_k(t)
		e^{2\pi i k\cdot x/L}
		\end{align*}
		gives a periodic scalar function satisfying
	$
		\boldsymbol{v}=\nabla\psi.
	$
		
		\smallskip
		\noindent
		\textbf{Step 5:} \emph{Recovering the scalar equation.}
			Since
	$
		\boldsymbol{v}=\nabla\psi,
	$
		we have
		\begin{align*}
			\boldsymbol{v}_t=\nabla\psi_t,
		\ \
		\Delta^2\boldsymbol{v}=\nabla\Delta^2\psi,
		\ \
		\Delta\boldsymbol{v}=\nabla\Delta\psi \ \text{ and }\ 	(\boldsymbol{v}\cdot\nabla)\boldsymbol{v}
		=
		\frac12\nabla|\nabla\psi|^2.
		\end{align*}
		Therefore the vector equation \eqref{eq:vector} becomes
		\begin{align*}
		\nabla\psi_t
		+\nabla\Delta^2\psi
		+\nabla\Delta\psi
		+\frac12\nabla|\nabla\psi|^2
		=
		\lambda_{\boldsymbol{v}}(t)\nabla\psi.
		\end{align*}
		Equivalently, we have 
		\begin{align*}
		\nabla
		\left[
		\psi_t+\Delta^2\psi+\Delta\psi
		+\frac12|\nabla\psi|^2
		-\lambda_{\boldsymbol{v}}(t)\psi
		\right]
		=0.
			\end{align*}
		Hence there exists a function $c(t)$, independent of $x$, such that
			\begin{align*}
		\psi_t+\Delta^2\psi+\Delta\psi
		+\frac12|\nabla\psi|^2
		=
		\lambda_{\boldsymbol{v}}(t)\psi+c(t).
			\end{align*}
		Since a potential is determined only up to an additive function of
		time, define
		\begin{align*}
		\widetilde\psi(t,x)
		=
		\psi(t,x)+q(t),
		\end{align*}
		where $q$ solves
		\begin{align*}
		q'(t)-\lambda_{\boldsymbol{v}}(t)q(t)=-c(t).
		\end{align*}
		Then
		$
		\nabla\widetilde\psi=\nabla\psi=\boldsymbol{v},
	$
		and
		\begin{align*}
			&\widetilde\psi_t
			+\Delta^2\widetilde\psi
			+\Delta\widetilde\psi
			+\frac12|\nabla\widetilde\psi|^2
			-\lambda(t)\widetilde\psi=
			c(t)+q'(t)-\lambda_{\boldsymbol{v}}(t)q(t)
			=0.
		\end{align*}
		Thus \eqref{eq:scalar} is satisfied  and by Step 2, we infer 
	$
		\|\boldsymbol{v}(t)\|_{L^2}=1,
	$ for all $t\geq 0$. 
		Since $\boldsymbol{v}=\nabla\widetilde\psi$,
		$
			\|\nabla\widetilde\psi(t)\|_{L^2}=1$ for all $t\geq 0$. 
		Thus $\widetilde\psi$ satisfies the scalar constrained problem.
		
		Therefore, for $d=2,3$, the scalar constrained formulation and the
		vector formulation are equivalent on the class of periodic gradient
		fields satisfying \eqref{eqn-curl}
		with the constraint imposed only at the initial time:
		$
		\|\boldsymbol{v}_0\|_{L^2}=1.
	$
		The constraint is then preserved by the evolution.
	\end{proof}

The following result establishes the local well-posedness of the KS equation in dimensions $d=2,3$ using the Banach fixed-point theorem. A similar approach can be found in \cite{AmbroseMazzucato2019}.

	\begin{proof}[Proof of Theorem \ref{thm-local}]
		We proceed the local existence and uniqueness proof in several steps.

			\smallskip
		\noindent
		\textbf{Sep 1:} \emph{Local existence.} 
		Let us define 
		\begin{align}\label{eqn-space}
		\mathcal{X}_T:=C([0,T];H^2(\mathbb T_L^d;\mathbb{R}^d)),\ \
		\norm{\boldsymbol{v}}_{\mathcal{X}_T}=\sup_{0\le t\le T}\norm{\boldsymbol{v}(t)}_{H^2}.
		\end{align}
		Let us choose $R>0$ sufficiently large such that
		\begin{align*}
		C_1\|\boldsymbol{v}_0\|_{H^2}<\frac{R}{2},
		\end{align*}
		where $C_1>0$ is the constant in \eqref{eq:SA-H2-H2}, satisfying
		\begin{align*}
		\|S_A(t)\boldsymbol{v}\|_{H^2}
		\le C_1\|\boldsymbol{v}\|_{H^2},
		\  0\le t\le1.
		\end{align*}
		We then define
		\[
		\mathcal B_R
		:=
		\left\{
		\boldsymbol{v}\in\mathcal X_T:
		\|\boldsymbol{v}\|_{\mathcal X_T}\le R
		\right\}.
		\]
		Let us define
		\begin{align*}
		(\Phi \boldsymbol{v})(t)=S_A(t)\boldsymbol{v}_0+\int_0^tS_A(t-s)F(\boldsymbol{v}(s)) ds.
		\end{align*}
		For $\boldsymbol{v}\in\mathcal B_R$, it follows from \eqref{eq:SA-L2-H2} and \eqref{eq:FLip-2} that 
		\begin{align}
			\norm{\Phi \boldsymbol{v}(t)}_{H^2}&\leq \norm{S_A(t)\boldsymbol{v}_0}_{H^2}+\int_0^t\norm{S_A(t-s)F(\boldsymbol{v}(s))}_{H^2}ds\nonumber\\
			&\le C_1\norm{\boldsymbol{v}_0}_{H^2}
			+C_1C_RR\int_0^t(t-s)^{-1/2} ds\\
			&\le C_1\norm{\boldsymbol{v}_0}_{H^2}+2C_1C_RRt^{1/2}.
		\end{align}
		Therefore, we have 
		\begin{align}
		\norm{\Phi \boldsymbol{v}}_{\mathcal{X}_T}
		\leq
		C_1\norm{\boldsymbol{v}_0}_{H^2}
		+
		2C_1C_RRT^{1/2}.
		\end{align}
		Since
	$
		C_1\norm{\boldsymbol{v}_0}_{H^2}<\frac R2,
	$
		we choose $T>0$ sufficiently small that
	$
		2C_1C_RRT^{1/2}<\frac R2.
	$
		Then, we deduce 
	$
		\norm{\Phi \boldsymbol{v}}_{\mathcal{X}_T}\leq R,
	$
	so that 
	$
		\Phi(\mathcal B_R)\subseteq\mathcal B_R.
	$

		Let us take $\boldsymbol{v},\boldsymbol{w}\in\mathcal B_R$. Then, we infer 
		\begin{align}
		\Phi \boldsymbol{v}(t)-\Phi \boldsymbol{w}(t)
		=
		\int_0^t
		S_A(t-s)
		\bigl(F(\boldsymbol{v}(s))-F(\boldsymbol{w}(s))\bigr) ds.
		\end{align}
		Using \eqref{eq:SA-L2-H2} and \eqref{eq:FLip-1},  we obtain 
		\begin{align}\label{eqn-est-5}
			\norm{\Phi \boldsymbol{v}(t)-\Phi \boldsymbol{w}(t)}_{H^2}
			&\leq
			C_1C_{R}
			\int_0^t
			(t-s)^{-1/2}
			\norm{\boldsymbol{v}(s)-\boldsymbol{w}(s)}_{H^2} ds\\
			&\leq
			2C_1C_{R}t^{1/2}
			\norm{\boldsymbol{v}-\boldsymbol{w}}_{\mathcal{X}_T}.
		\end{align}
		Consequently, we deduce 
		\begin{align*}
		\norm{\Phi \boldsymbol{v}-\Phi \boldsymbol{w}}_{\mathcal{X}_T}
		\leq
		2C_1C_{R}T^{1/2}
		\norm{\boldsymbol{v}-\boldsymbol{w}}_{\mathcal{X}_T}.
		\end{align*}
		Choose $T$ smaller, if necessary, so that
	$
		2C_1C_{R}T^{1/2}<1.
	$
		Then $\Phi$ is a contraction on $\mathcal B_R$.
		
		By the \emph{Banach contraction mapping theorem}, there exists a unique
		$
		\boldsymbol{v}\in\mathcal B_R
	$
		such that
	$
		\Phi \boldsymbol{v}=\boldsymbol{v}.
	$
		Therefore
		\begin{equation}\label{eq:mild}
			\boldsymbol{v}(t)
			=
			S_A(t)\boldsymbol{v}_0
			+
			\int_0^tS_A(t-s)F(\boldsymbol{v}(s)) ds.
		\end{equation}
It follows form \eqref{eqn-space} that 
	$
		\boldsymbol{v}\in C([0,T];H^2(\mathbb T_L^d;\mathbb{R}^d)).
	$

			\smallskip
		\noindent
		\textbf{Sep 2:} \emph{Additional regularity.}  
			It remains to prove the additional regularity. By Step 1, the fixed point $\boldsymbol{v}$ belongs to the ball $\mathcal{B}_R$. Let us put 
	$$
		G=F(\boldsymbol{v})+\ell\boldsymbol{v}, 
	$$
	where $\ell$ is the same constant defined in \eqref{eqn-sel-b}.	Since
	$\boldsymbol{v}\in \mathcal{B}_R$, estimate \eqref{eq:FLip-2} yields
$
	F(\boldsymbol{v})\in L^2(0,T;L^2(\mathbb T_L^d;\mathbb{R}^d)).
$
	Thus, we infer that $G \in L^2(0,T;L^2(\mathbb T_L^d;\mathbb{R}^d))$ and
		\begin{align}
		\boldsymbol{v}_t(t)+B \boldsymbol{v}(t)=G(t) \ \text{ in } \ L^2(\mathbb T_L^d;\mathbb{R}^d),
		\end{align}
	for a.e. $t\in(0,T)$.	In Fourier variables,  writing
	$
	G(x,t)=\sum_{k\in\mathbb Z^d}\widehat{G}_k(t)e_k(x),
	\ 
	\widehat{G}_k(t)=(G(t),e_k),
$ and using 	$\boldsymbol{v}(x,t)=\sum_{k\in\mathbb Z^d}\widehat{\boldsymbol{v}}_k(t)e_k(x),$ the above equation leads to 
		\begin{align*}
		\widehat{\boldsymbol{v}}_k'(t)+\beta_k\widehat{\boldsymbol{v}}_k(t)=\widehat{G}_k(t),
		\end{align*}
	where $\beta_k$ is defined in \eqref{eqn-betak}. 	Multiplying the equation by
	$\beta_k\overline{\widehat{\boldsymbol{v}}}_k$, taking the real part,
	and then applying the Cauchy--Schwarz and Young inequalities, we obtain
		\begin{align}
		\frac12\frac d{dt}\bigl(\beta_k|\widehat{\boldsymbol{v}}_k(t)|_{\mathbb{C}^d}^2\bigr)
		+\beta_k^2|\widehat{\boldsymbol{v}}_k(t)|_{\mathbb{C}^d}^2 &=\beta_k\operatorname{Re}\left(\widehat{G}_k(t)\overline{\widehat{\boldsymbol{v}}}_k(t)\right)\leq \beta_k| \widehat{G}_k(t)||\overline{\widehat{\boldsymbol{v}}}_k(t)|\nonumber\\&\leq  \frac12 \beta_k^2|\widehat{\boldsymbol{v}}_k(t)|_{\mathbb{C}^d}^2+
		\frac12|\widehat{G}_k(t)|_{\mathbb{C}^d}^2,
		\end{align}
	for a.e. $t\in[0,T]$.	After integration over time and summation, we deduce 
		\begin{align*}
		\sup_{0\le t\le T}\sum_{k\in\mathbb Z^d}\beta_k|\widehat{\boldsymbol{v}}_k(t)|_{\mathbb{C}^d}^2
		+ \int_0^T\sum_{k\in\mathbb Z^d}\beta_k^2|\widehat{\boldsymbol{v}}_k(t)|_{\mathbb{C}^d}^2 dt
		\le
		\sum_{k\in\mathbb Z^d}\beta_k|\widehat{\boldsymbol{v}}_{0,k}|_{\mathbb{C}^d}^2+\norm G_{L^2(0,T;L^2)}^2,
		\end{align*}
	where we have used Parseval's identity.	Since $\beta_k\asymp(1+|\xi_k|^2)^2$ (see \eqref{eqn-est-1}), we deduce 
	$
		\sum_{k\in\mathbb Z^d}\beta_k|\widehat{\boldsymbol{v}}_k|_{\mathbb{C}^d}^2
		\asymp
		\norm{\boldsymbol{v}}_{H^2}^2
	$
		and
	$
		\sum_{k\in\mathbb Z^d}\beta_k^2|\widehat{\boldsymbol{v}}_k|_{\mathbb{C}^d}^2
		\asymp
		\norm{\boldsymbol{v}}_{H^4}^2,
	$ and using the fact that $\boldsymbol{v}_{0}\in H^2(\mathbb T_L^d;\mathbb{R}^d)$, we  infer from the above estimate that 
		\begin{align*}
		\boldsymbol{v}\in L^2(0,T;H^4(\mathbb T_L^d;\mathbb{R}^d))
		\quad\text{and}\quad
		B\boldsymbol{v}\in L^2(0,T;L^2(\mathbb T_L^d;\mathbb{R}^d)).
		\end{align*}
		Therefore, we get 
	$
		\boldsymbol{v}_t=G-B\boldsymbol{v}\in L^2(0,T;L^2(\mathbb T_L^d;\mathbb{R}^d)),
	$
		and hence
	$$
	\boldsymbol{v}		\in
	C([0,T];H^2(\mathbb T_L^d;\mathbb{R}^d))
	\cap
	H^1(0,T;L^2(\mathbb T_L^d;\mathbb{R}^d))
	\cap
	L^2(0,T;H^4(\mathbb T_L^d;\mathbb{R}^d)).
	$$
		Since \(\boldsymbol{v}_0\in\mathcal{M}\), an argument analogous to that used in Step 2 of the proof of Proposition \ref{prop-equivalence-2} shows that
		$ \boldsymbol{v}(t)\in\mathcal{M} $ for all $t\in[0,T]$. Thus, the manifold \(\mathcal{M}\) is invariant under the evolution.
		
			\smallskip
		\noindent
		\smallskip
	\noindent
	\textbf{Step 3:} \emph{Continuous dependence and uniqueness.}
	Let $\boldsymbol{v}$ and $\boldsymbol{w}$ be two solutions of
	\eqref{eq:KSE}, corresponding, respectively, to the initial data
	$\boldsymbol{v}_0$ and $\boldsymbol{w}_0$, such that
	\begin{align*}
		\boldsymbol{v},\boldsymbol{w}
		\in
		C([0,T];H^2(\mathbb T_L^d;\mathbb{R}^d)\cap\mathcal{M})
		\cap
		H^1(0,T;L^2(\mathbb T_L^d;\mathbb{R}^d))
		\cap
		L^2(0,T;H^4(\mathbb T_L^d;\mathbb{R}^d)).
	\end{align*}
	Fix $R>0$ such that
	\begin{align*}
		\|\boldsymbol{v}_0\|_{H^2}
		+
		\|\boldsymbol{w}_0\|_{H^2}
		\le R.
	\end{align*}
	For $T$ sufficiently small, both solutions remain in a fixed bounded
	subset of $C([0,T];H^2(\mathbb T_L^d;\mathbb{R}^d))$. Their difference satisfies
	\begin{align*}
	\begin{aligned}
		\boldsymbol{v}(t)-\boldsymbol{w}(t)
		=
		S_A(t)(\boldsymbol{v}_0-\boldsymbol{w}_0)+
		\int_0^t
		S_A(t-s)
		\bigl(F(\boldsymbol{v}(s))-F(\boldsymbol{w}(s))\bigr) ds,
	\end{aligned}
	\end{align*}
for all $t\in[0,T]$. Calculations similar to \eqref{eq:SA-H2-H2} and \eqref{eqn-est-5} yield 
	\begin{align}\label{eqn-est-4}
	\norm{\boldsymbol{v}-\boldsymbol{w}}_{\mathcal{X}_T}
	\leq
	C_1\norm{\boldsymbol{v}_0-\boldsymbol{w}_0}_{H^2}
	+
	2C_1C_RT^{1/2}
	\norm{\boldsymbol{v}-\boldsymbol{w}}_{\mathcal{X}_T}.
	\end{align}
	Choose $T$ so that
	$
	2C_1C_RT^{1/2}\leq\frac12. 
$
Then, from \eqref{eqn-est-4}, we deduce 
	\begin{align}\label{eqn-est-6}
	\norm{\boldsymbol{v}-\boldsymbol{w}}_{\mathcal{X}_T}
	\leq
	2C_1\norm{\boldsymbol{v}_0-\boldsymbol{w}_0}_{H^2}.
	\end{align}
	Hence the solution map is locally Lipschitz from $H^2$ into
	$C([0,T];H^2(\mathbb T_L^d;\mathbb{R}^d))$.
	
	Let $[0,T_*]$ be any common interval of existence. Since
$
	\boldsymbol{v},\boldsymbol{w}\in C([0,T_*];H^2(\mathbb T_L^d;\mathbb{R}^d)),
$
	their $H^2$ norms are bounded on $[0,T_*]$. Hence there exists
	$R>0$ such that
	\begin{align*}
	\sup_{0\leq t\leq T_*}
	\left(
	\|\boldsymbol{v}(t)\|_{H^2}+\|\boldsymbol{w}(t)\|_{H^2}
	\right)\leq R.
	\end{align*}
	Choose $\delta>0$ such that
$
	2C_1C_R\delta^{1/2}<1.
$
For $\boldsymbol{v}_0= \boldsymbol{w}_0$, the contraction estimate \eqref{eqn-est-4} shows that $\boldsymbol{v}=\boldsymbol{w}$ on $[0,\delta]$. Suppose now that $\boldsymbol{v}=\boldsymbol{w}$ on $[0,t_j]$. Using $\boldsymbol{v}(t_j)=\boldsymbol{w}(t_j)$
	as the initial datum and applying the same estimate on
	$[t_j,t_j+\delta]$, we obtain
$
	\boldsymbol{v}=\boldsymbol{w}
	\ \text{ on }\ [t_j,t_j+\delta].
$
	Repeating finitely many times covers $[0,T_*]$, so that 
$
		\boldsymbol{v}=	\boldsymbol{w} \ \text{ on }\ [0,T_*].
$

	\smallskip
	\noindent
	\textbf{Step 4:} \emph{Continuation criterion.}
	Let $T_{\max}$ denote the maximal existence time.
	Suppose that
$
	T_{\max}<\infty
$
	but
$
	\sup_{0\leq t<T_{\max}}
	\norm{	\boldsymbol{v}(t)}_{H^2}
	\leq R.
$
	The fixed-point construction above depends only on the bound $R$ and
	the torus size $L$. Hence there exists
$
	\tau_R>0
$
	such that every initial datum $	\boldsymbol{v}_0\in H^2$ satisfying
$
	\norm{	\boldsymbol{v}_0}_{H^2}\leq R
$
	generates a unique solution on an interval of length $\tau_R$.
		Choose
$
	t_0<T_{\max}
$
	so close to $T_{\max}$ that
$
	T_{\max}-t_0<\tau_R.
$
	Since
$
	\norm{	\boldsymbol{v}(t_0)}_{H^2}\leq R,
$
	the local existence result applied to the initial datum $	\boldsymbol{v}(t_0)$ gives
	a solution on
$
	[t_0,t_0+\tau_R].
$
	By uniqueness, this solution coincides with the original solution on
	$
	[t_0,T_{\max}).
$
	But
$
	t_0+\tau_R>T_{\max},
$
	so the new solution extends the original solution beyond $T_{\max}$.
	This contradicts the maximality of $T_{\max}$. Therefore, we conclude 
	$$	T_{\max}<\infty
		\ \Longrightarrow\ 
		\limsup_{t\uparrow T_{\max}}
		\norm{	\boldsymbol{v}(t)}_{H^2}
		=
		\infty,
		$$
		which completes the proof. 
	\end{proof}

		\section{Proof of Theorem \ref{thm-main}}\label{sec-5}
	
Based on the local well-posedness established in Theorem \ref{thm-local}, we prove in this section the global well-posedness of the two-dimensional constrained KS equation \eqref{eq:KSE}.

	\begin{proof}[Proof of Theorem \ref{thm-main}]
The local existence of a strong solution to problem \eqref{eq:KSE} is established in Theorem \ref{thm-local}. We have also shown that this local solution is unique and can be continued as long as its $H^2$ norm remains finite. Consequently, the global $H^2$ estimate proved below implies that $T=T_{\max}=\infty.$
The preservation of the $L^2$ constraint was established in Step 2 of the proof of Proposition \ref{prop-equivalence-2}. Therefore, we have 
		\begin{equation}\label{eq:constraint}
				\norm{\boldsymbol{v}(t)} _{L^2}=1\ \text{	for all }\ 0\leq t<T. 
		\end{equation}
		In particular, we derive 
		\begin{equation}\label{eq:orthogonal}
			(\boldsymbol{v}(t),\boldsymbol{v}_t(t))
			=
			\frac12\frac{d}{dt}\norm{\boldsymbol{v}(t)} _{L^2}^2
			=0,
		\end{equation}
		for a.e. $t\in(0,T)$. 
	From the local existence result (Theorem \ref{thm-local}), we have the regularity result given in \eqref{eqn-reg-1}. 
	We may therefore perform the following calculations.
			Taking the $L^2$-inner product of \eqref{eq:KSE} with $\boldsymbol{v}_t(\cdot)$ gives
		\begin{align}\label{eqn-ener-1}
	&	\norm{\boldsymbol{v}_t(t)} _{L^2}^2
		+
		(\Delta^2\boldsymbol{v}(t),\boldsymbol{v}_t(t))
		+
		(\Delta \boldsymbol{v}(t),\boldsymbol{v}_t(t))
		+
		((\boldsymbol{v}(t)\cdot\nabla)\boldsymbol{v}(t),\boldsymbol{v}_t(t))
	\nonumber\\&	=
		\lambda_{\boldsymbol{v}}(\boldsymbol{v}(t),\boldsymbol{v}_t(t)),
		\end{align}
	for a.e. $t\in(0,T)$.	Using  \eqref{eq:orthogonal}, we infer for a.e. $t\in(0,T)$  that 
		\begin{align*}
		\lambda_{\boldsymbol{v}}(\boldsymbol{v}(t),\boldsymbol{v}_t(t))=0.
		\end{align*}
			By periodic integration by parts, we have 
		\begin{align}
			(\Delta^2\boldsymbol{v},\boldsymbol{v}_t)
			=
			(\Delta \boldsymbol{v},\Delta \boldsymbol{v}_t)
			=
			\frac12\frac{d}{dt}
			\norm{\Delta \boldsymbol{v}} _{L^2}^2\ \text{ and }\ 
			(\Delta \boldsymbol{v},\boldsymbol{v}_t)
			=
			-(\nabla \boldsymbol{v},\nabla \boldsymbol{v}_t)=
			-\frac12\frac{d}{dt}
			\norm{\nabla \boldsymbol{v}} _{L^2}^2.
		\end{align}
		Consequently, we deduce from \eqref{eqn-ener-1} that 
		\begin{equation}\label{eq:energyidentity}
			\frac12\frac{d}{dt}
			\left(
			\norm{\Delta \boldsymbol{v}(t)} _{L^2}^2
			-
			\norm{\nabla \boldsymbol{v}(t)} _{L^2}^2
			\right)
			+
			\norm{\boldsymbol{v}_t(t)} _{L^2}^2
			=
			-((\boldsymbol{v}(t)\cdot\nabla)\boldsymbol{v}(t),\boldsymbol{v}_t(t)),
		\end{equation}
	for a.e. $t\in(0,T)$.	Let us 	define
		\begin{align}\label{eqn-def-e}
		E(t)
		:=
		\frac12
		\left(
		\norm{\Delta \boldsymbol{v}(t)} _{L^2}^2
		-
		\norm{\nabla \boldsymbol{v}(t)} _{L^2}^2
		\right), \ t\in[0,T]. 
		\end{align}
		Then, it follows from  \eqref{eq:energyidentity} that 
		\begin{equation}\label{eq:Eidentity}
			E'(t)+\norm{\boldsymbol{v}_t(t)} _{L^2}^2
			=
			-((\boldsymbol{v}(t)\cdot\nabla)\boldsymbol{v}(t),\boldsymbol{v}_t(t)),
		\end{equation}
		for a.e. $t\in[0,T]$.  In order to estimate the nonlinear term on the right hand side of \eqref{eq:Eidentity}, we use the following two-dimensional Agmon inequality: 
		\begin{equation*}
			\norm{\boldsymbol{v}}_{L^\infty}
			\leq
			C_L
			\norm{\boldsymbol{v}} _{L^2}^{1/2}
			\norm{\boldsymbol{v}}_{H^2}^{1/2}.
		\end{equation*}
		By \eqref{eq:constraint}, we know that 
	$
		\norm{\boldsymbol{v}} _{L^2}=1,
	$
	and hence 
		\begin{align}\label{eq:Agmon}
		\norm{\boldsymbol{v}}_{L^\infty}
		\leq
		C_L\norm{\boldsymbol{v}}_{H^2}^{1/2}.
		\end{align}
		On the other hand, periodic integration by parts gives
		\begin{align}
		\norm{\nabla \boldsymbol{v}} _{L^2}^2
		=
		-(\boldsymbol{v},\Delta \boldsymbol{v})
		\leq
		\norm{\boldsymbol{v}} _{L^2}\norm{\Delta \boldsymbol{v}} _{L^2}
		=
		\norm{\Delta \boldsymbol{v}} _{L^2},
		\end{align}
	where we have used  \eqref{eq:constraint}. 	Thus, it is immediate that 
		\begin{equation}\label{eq:grad-est}
			\norm{\nabla \boldsymbol{v}} _{L^2}
			\leq
			\norm{\Delta \boldsymbol{v}} _{L^2}^{1/2}.
		\end{equation}
			Moreover, we infer from \eqref{eqn-est-12} that 
		\begin{align}\label{eqn-ener-6}
		\norm{\boldsymbol{v}}_{H^2}
		\leq
		C_L
		\left(
		\norm{\boldsymbol{v}} _{L^2}+\norm{\Delta \boldsymbol{v}} _{L^2}
		\right)
		=
		C_L(1+\norm{\Delta \boldsymbol{v}} _{L^2}).
		\end{align}
		Therefore, we deduce from \eqref{eq:Agmon} that 
	\begin{align}\label{eqn-ener-2}
		\norm{\boldsymbol{v}}_{L^{\infty}}
		\leq
		C_L
		(1+\norm{\Delta \boldsymbol{v}} _{L^2})^{1/2}.
		\end{align}
		It follows by using H\"older's inequality, \eqref{eq:grad-est} and \eqref{eqn-ener-2}  that
		\begin{align*}
		\begin{aligned}
			\norm{(\boldsymbol{v}\cdot\nabla)\boldsymbol{v}} _{L^2}
			&\leq
			\norm{\boldsymbol{v}}_{L^\infty}\norm{\nabla \boldsymbol{v}} _{L^2}\leq
			C_L
			(1+\norm{\Delta \boldsymbol{v}} _{L^2})^{1/2}
			\norm{\Delta \boldsymbol{v}} _{L^2}^{1/2}.
		\end{aligned}
		\end{align*}
		For every $a\geq0$, we know that 
	$
		(1+a)^{1/2}a^{1/2}
		\leq C(1+a).
	$
		Consequently, we have 
		\begin{equation}\label{eq:nonlinear}
				\norm{(\boldsymbol{v}\cdot\nabla)\boldsymbol{v}} _{L^2}
				\leq
				C_L(1+\norm{\Delta \boldsymbol{v}} _{L^2}).
		\end{equation}
	Therefore, we estimate the right hand side of \eqref{eq:Eidentity} as 
		\begin{align}\label{eqn-cruicial}
			|((\boldsymbol{v}\cdot\nabla)\boldsymbol{v},\boldsymbol{v}_t)|
			\leq
			C_L(1+\norm{\Delta \boldsymbol{v}} _{L^2})\norm{\boldsymbol{v}_t} _{L^2}\leq
			\frac12\norm{\boldsymbol{v}_t} _{L^2}^2
			+
			C_L(1+\norm{\Delta \boldsymbol{v}} _{L^2}^2),
		\end{align}
		where Young's inequality was used in the last step.  The estimate \eqref{eqn-cruicial} is the central estimate underlying the proof of this theorem.
			Substitution into \eqref{eq:Eidentity} gives
		\begin{equation}\label{eq:Eineq}
			E'(t)
			+
			\frac12\norm{\boldsymbol{v}_t(t)} _{L^2}^2
			\leq
			C_L
			\left(
			1+\norm{\Delta \boldsymbol{v}(t)} _{L^2}^2
			\right),
		\end{equation}
		foe a.e. $t\in[0,T]$. 
	
		Let us now set 
		\begin{align*}
		Z(t)=\norm{\Delta \boldsymbol{v}(t)} _{L^2}, \ t\in[0,T]. 
		\end{align*}
		From \eqref{eq:grad-est}, we infer 
		\begin{align*}
		\norm{\nabla \boldsymbol{v}(t)} _{L^2}^2\leq Z(t), \  t\in[0,T]. 
		\end{align*}
		Therefore,  the definition of $E(\cdot)$ given in \eqref{eqn-def-e} yields  for a.e. $t\in[0,T]$ that 
		\begin{align}
		E(t)
		=
		\frac12
		\left(
		Z(t)^2-\norm{\nabla \boldsymbol{v}(t)} _{L^2}^2
		\right)
		\geq
		\frac12(Z(t)^2-Z(t)) =	\frac12\left(Z(t)-\frac12\right)^2-\frac18.
		\end{align}
		Therefore, we find 
		\begin{align*}
		F(t)=E(t)+\frac18 \geq
		\frac12
		\left(
		Z(t)-\frac12
		\right)^2\geq 0. 
		\end{align*}
	Consequently, we deduce 
		\begin{align*}
		Z(t)
		\leq
		\frac12+\sqrt{2F(t)} \ \Longrightarrow \  	Z(t)^2
		\leq
		4F(t)+\frac12.
		\end{align*}
		Therefore, we arrive at 
		\begin{equation}\label{eq:Zcontrol}
			\norm{\Delta \boldsymbol{v}(t)} _{L^2}^2
			\leq
			C(1+F(t)), \ t\in [0,T]. 
		\end{equation}
		Since $F'=E'$, \eqref{eq:Eineq} and \eqref{eq:Zcontrol} imply
		\begin{align*}
		F'(t)
		\leq
		C_L(1+F(t)),
		\end{align*}
	for a.e. $t\in[0,T]$.	Hence, an application of Gr\"onwall's inequality leads to 
		\begin{equation}\label{eq:Fgronwall}
			F(t)+1
			\leq
			(F(0)+1)e^{C_Lt}, \ \text{ for all } \ t\in[0,T]. 
		\end{equation}
			In particular, for every $T>0$,  it follows from \eqref{eq:Zcontrol} that
		\begin{align}\label{eqn-uni-bound}
		\sup_{0\leq t\leq T}F(t)\leq C_T  \ \Longrightarrow \   \sup_{0\leq t\leq T}
		\norm{\Delta \boldsymbol{v}(t)} _{L^2}^2
		\leq C_T.
		\end{align}
The estimate \eqref{eq:grad-est} implies 
		\begin{align*}
		\norm{\nabla \boldsymbol{v}(t)} _{L^2}^2
		\leq
		\norm{\Delta \boldsymbol{v}(t)} _{L^2}
		\leq C_T \ \text{ for all }\ t\in[0,T]. 
		\end{align*}
		Together with \eqref{eq:constraint}, the above two estimates  give
		\begin{equation}\label{eq:H2bound}
				\sup_{0\leq t\leq T}
				\norm{\boldsymbol{v}(t)}_{H^2}
				\leq C_T,
		\end{equation}
		which proves the global $H^2$ estimate.

		Suppose, for contradiction, that
	$
		T_{\max}<\infty.
	$
		The estimate \eqref{eq:H2bound}, applied on every interval
		$[0,T]$ with $T<T_{\max}$, gives
		\begin{align*}
		\sup_{0\leq t<T_{\max}}
		\norm{\boldsymbol{v}(t)}_{H^2}
		\leq C_{T_{\max}}<\infty.
		\end{align*}
		This contradicts the continuation criterion \eqref{eqn-max}.
		Therefore
	$
			T_{\max}=\infty,
	$
	and	the solution exists globally.

		Integrating \eqref{eq:Eineq} from $0$ to $T$ gives
		\begin{align}
		\frac12\int_0^T\norm{\boldsymbol{v}_t(t)} _{L^2}^2 dt
		\leq
		E(0)-E(T)
		+
		C_LT
		+
		C_L\int_0^T\norm{\Delta \boldsymbol{v}(t)} _{L^2}^2 dt.
		\end{align}
		Since
		$
		E(T)\geq-\frac18
		$
		and
		\eqref{eqn-uni-bound} holds, 
		we obtain
		\begin{equation}\label{eq:vtbound}
				\int_0^T\norm{\boldsymbol{v}_t(t)} _{L^2}^2 dt
				\leq C_T,
		\end{equation}
	so that 
	$
		\boldsymbol{v}_t\in L^2(0,T;L^2(\mathbb T_L^2;\mathbb{R}^2)).
	$
	From \eqref{eq:lambda-1}, we observe for all $t\in[0,T]$ that 
		\begin{align}\label{eqn-lambda-est}
			|\lambda_{\boldsymbol{v}}(t)|
			\leq
			\norm{\Delta \boldsymbol{v}(t)} _{L^2}^2+\norm{\nabla \boldsymbol{v}(t)} _{L^2}^2+
			\frac12
			\left|
			\int_{\mathbb T_L^2}
			|\boldsymbol{v}(t)|^2\operatorname{div}\boldsymbol{v}(t) dx
			\right|.
		\end{align}
		For the last term in the right hand side of \eqref{eqn-lambda-est}, we use H\"older's and Gagliardo--Nirenberg's inequalities to estimate it as
		\begin{align}\label{eqn-ener-3}
			\left|
			\int |\boldsymbol{v}|^2\operatorname{div}\boldsymbol{v} dx
			\right|
		&	\leq\sqrt{2}
			\norm{\boldsymbol{v}}_{L^4}^2\norm{\nabla \boldsymbol{v}} _{L^2}\leq C_L\|\boldsymbol{v}\|_{L^2}\left(\|\boldsymbol{v}\|_{L^2}+\|\nabla \boldsymbol{v}\|_{L^2}\right) \|\nabla \boldsymbol{v}\|_{L^2} \nonumber\\&\leq C_L\left(1+ \|\nabla \boldsymbol{v}\|_{L^2}^2\right),
		\end{align}
	where we have used the fact that $ \boldsymbol{v}\in\mathcal{M}$. 	
		Therefore using \eqref{eqn-uni-bound} and \eqref{eqn-ener-3} in \eqref{eqn-lambda-est}, we arrive at 
		\begin{align}\label{eqn-ener-4}
			\sup_{0\leq t\leq T}|\lambda_{\boldsymbol{v}}(t)|
			\leq C_T.
	\end{align}

		Rearranging \eqref{eq:KSE}, we obtain 
		\begin{align}\label{eqn-ener-5}
		\Delta^2\boldsymbol{v}
		=
		-\boldsymbol{v}_t-\Delta \boldsymbol{v}-(\boldsymbol{v}\cdot\nabla)\boldsymbol{v}-\lambda_{\boldsymbol{v}} \boldsymbol{v}.
		\end{align}
		By \eqref{eqn-uni-bound} and  \eqref{eq:vtbound}, we deduce 
		\begin{align*}
		\boldsymbol{v}_t\in L^2(0,T;L^2(\mathbb T_L^2;\mathbb{R}^2)) \ \text{ and }\ 	\Delta \boldsymbol{v}\in L^\infty(0,T;L^2(\mathbb T_L^2;\mathbb{R}^2))
		\subset L^2(0,T;L^2(\mathbb T_L^2;\mathbb{R}^2)).
		\end{align*}
		Furthermore, by \eqref{eq:nonlinear} and the $H^2$-estimate \eqref{eqn-uni-bound}, we get 
		\begin{align*}
		(	\boldsymbol{v}\cdot\nabla)	\boldsymbol{v}
		\in L^\infty(0,T;L^2(\mathbb T_L^2;\mathbb{R}^2))
		\subset L^2(0,T;L^2(\mathbb T_L^2;\mathbb{R}^2)).
		\end{align*}
		Finally, the estimate \eqref{eqn-ener-4} implies 
		\begin{align*}
		\lambda_{\boldsymbol{v}}	\boldsymbol{v}
		\in L^\infty(0,T;L^2(\mathbb T_L^2;\mathbb{R}^2))
		\subset L^2(0,T;L^2(\mathbb T_L^2;\mathbb{R}^2)).
		\end{align*}
		Hence, from \eqref{eqn-ener-5}, we infer 
			\begin{align*}
		\Delta^2\boldsymbol{v}\in L^2(0,T;L^2(\mathbb T_L^2;\mathbb{R}^2)).
	\end{align*}
		Consequently, by using \eqref{eq:elliptic}, we get 
	$
		\boldsymbol{v}	\in
			L^2(0,T;H^4(\mathbb T_L^2;\mathbb R^2)).
	$
		Since $T>0$ is arbitrary, we deduce 
	\begin{align*}
			\boldsymbol{v}	 \in
		L^2_{\mathrm{loc}}
		(0,\infty;H^4(\mathbb T_L^2;\mathbb R^2)).
		\end{align*}

		The local solution satisfies
		\begin{align*}
		\boldsymbol{v}\in
		C([0,T_{\max});H^2(\mathbb T_L^2;\mathbb{R}^2))
		\cap
		H^1(0,T_{\max};L^2(\mathbb T_L^2;\mathbb{R}^2))
		\end{align*}
	and we have proven that 
	$
		T_{\max}=\infty.
	$
		Therefore, we have 
		\begin{align*}
		\boldsymbol{v}\in
		C([0,\infty);H^2(\mathbb T_L^2;\mathbb{R}^2))
		\cap
	H^1_{\mathrm{loc}}(0,\infty;L^2(\mathbb T_L^2;\mathbb{R}^2)).
		\end{align*}
		Uniqueness follows from the local well-posedness theorem (Step 3 in the proof of Theorem \ref{thm-local}).
		Indeed, if $\boldsymbol{v}_1$ and $\boldsymbol{v}_2$ are two solutions of \eqref{eq:KSE} with the same initial datum,
		then local uniqueness gives
	$
		\boldsymbol{v}_1=\boldsymbol{v}_2
	$
		on some interval $[0,T_1]$. At time $T_1$, both solutions have the same
		value in $H^2(\mathbb T_L^2;\mathbb{R}^2)$, and the same local uniqueness theorem can be applied again.
		Repeating this argument on every finite interval gives
		$
		\boldsymbol{v}_1=\boldsymbol{v}_2
		\ \text{ on }\ [0,\infty).
		$
			Finally, the constraint \eqref{eq:constraint} holds on the whole global
		interval. Thus
	$
		\norm{\boldsymbol{v}(t)} _{L^2}=1
		\ \text{ for every }\ t\geq0.
		$
		This completes the proof.
	\end{proof}

	\begin{remark}
		Alternatively, by applying Hölder's inequality together with the Gagliardo--Nirenberg interpolation inequality, we obtain an estimate for \(\|(\boldsymbol{v}\cdot\nabla)\boldsymbol{v}\|_{L^2}\) that coincides with the bound in \eqref{eq:nonlinear}.  For $2<p\leq \infty$, using  \eqref{eq:constraint} and \eqref{eq:grad-est}, we have 
		\begin{align*}
				\norm{(\boldsymbol{v}\cdot\nabla)\boldsymbol{v}} _{L^2}&\leq \|\boldsymbol{v}\|_{L^p}\|\nabla\boldsymbol{v}\|_{L^{\frac{2p}{p-2}}}\nonumber\\&\leq C_L\|\boldsymbol{v}\|_{L^2}^{\frac2p}\left(\|\boldsymbol{v}\|_{L^2}+\|\nabla \boldsymbol{v}\|_{L^{2}}\right)^{\frac{p-2}{p}}\|\nabla \boldsymbol{v}\|_{L^2}^{\frac{p-2}{p}}\left(\|\boldsymbol{v}\|_{L^2}+\|\Delta \boldsymbol{v}\|_{L^{2}}\right)^{\frac2p}
				\nonumber\\&= C_L\left(1+\|\nabla \boldsymbol{v}\|_{L^{2}}\right)^{\frac{p-2}{p}} \|\nabla \boldsymbol{v}\|_{L^2}^{\frac{p-2}{p}}\left(1+\|\Delta \boldsymbol{v}\|_{L^{2}}\right)^{\frac2p}
				\nonumber\\&\leq C_L \left(1+\|\Delta \boldsymbol{v}\|_{L^{2}}^{\frac12}\right)^{\frac{p-2}{p}} \|\Delta \boldsymbol{v}\|_{L^{2}}^{\frac{p-2}{2p}}\left(1+\|\Delta \boldsymbol{v}\|_{L^{2}}\right)^{\frac2p}\nonumber\\&\leq C_L\left(1+\|\Delta \boldsymbol{v}\|_{L^{2}}\right).
		\end{align*}
	\end{remark}

	\begin{remark}\label{rem-3D}
The main difficulty in three dimensions lies in obtaining a suitable estimate for the right-hand side of \eqref{eq:Eidentity}. Indeed, the three-dimensional Agmon inequality gives
		\begin{equation*}
			\norm{\boldsymbol{v}}_{L^\infty}
			\leq
			C_L
			\norm{\boldsymbol{v}} _{L^2}^{1/4}
			\norm{\boldsymbol{v}}_{H^2}^{3/4}\leq  C_L
			\norm{\boldsymbol{v}}_{H^2}^{3/4}\leq C_L(1+\|\Delta\boldsymbol{v}\|_{L^2})^{3/4},
		\end{equation*}
		where we have used the fact that $\boldsymbol{v}\in\mathcal{M}$ and \eqref{eqn-ener-6}. A calculation similar to \eqref{eq:nonlinear} yields 
		\begin{align*}
			\norm{(\boldsymbol{v}\cdot\nabla)\boldsymbol{v}} _{L^2}
			\leq	C_L(1+\norm{\Delta \boldsymbol{v}} _{L^2})^{5/4}.
		\end{align*}
		However, this estimate is not sufficient to close the energy inequality \eqref{eq:Eineq} (see \eqref{eq:Eidentity} also). Therefore, the global solvability of the three-dimensional constrained KS equation is beyond the scope of the present work and is left for future investigation.
	\end{remark}

		\smallskip\noindent
		\textbf{Acknowledgments:} 
		The second author would like to thank Mr. Sagar Gautam for helpful initial discussions concerning the model.	H. Tiwari wants to thank University Grants Commission (UGC),	Govt. of India for financial assistance.  Support for M. T. Mohan's research received from the National Board of Higher Mathematics (NBHM), Department of Atomic Energy, Government of India (Project No. 02011/13/2025/NBHM(R.P)/R\&D II/1137).

				\smallskip\noindent	\textbf{Declarations:} 
				
				\noindent 	\textbf{Ethical Approval:}   Not applicable 
				
				\noindent  \textbf{Competing interests: } The authors declare no competing interests. 
				
				\noindent  \textbf{Conflict of interest: }On behalf of all authors, the corresponding author states that there is no conflict of interest.
				
				\noindent 	\textbf{Authors' contributions:} All authors have contributed equally. 
				
				\noindent 	\textbf{Availability of data and materials:} Not applicable. 
				
				\bibliographystyle{plain}
				\bibliography{Harsh}
				
				\end{document}